\documentclass[11pt]{amsart}
\usepackage{amssymb, amsmath, amsthm, mathrsfs}
\usepackage[all]{xy}
\usepackage[totalwidth=450pt, totalheight=650pt]{geometry}
\usepackage{comment}
\usepackage{enumitem}

\newtheorem{theorem}{Theorem}[section]
\newtheorem*{theorem*}{Theorem}
\newtheorem{corollary}[theorem]{Corollary}
\newtheorem*{corollary*}{Corollary}
\newtheorem{lemma}[theorem]{Lemma}
\newtheorem{proposition}[theorem]{Proposition}
\newtheorem{definition}[theorem]{Definition}
\newtheorem{remark}[theorem]{Remark}

\newcommand{\N}{\mathbb{N}}

\newcommand{\R}{\mathbb{R}}
\newcommand{\C}{\mathbb{C}}

\newcommand{\ie}{i.e. }

\newcommand{\cstar}{C$^*$}

\newcommand{\continuous}{\mathrm{C}}

\newcommand{\id}{\mathrm{id}}

\newcommand{\mat}{\mathrm{M}}

\newcommand{\sa}{\mathrm{sa}}

\newcommand{\unitaries}{\mathrm{U}}

\newcommand{\hess}{\mathrm{H}}
\newcommand{\bhess}{\mathrm{BH}}

\newcommand{\diag}{\mathrm{diag}}

\renewcommand{\epsilon}{\varepsilon}
\renewcommand{\phi}{\varphi}
\renewcommand{\emptyset}{\varnothing}
\renewcommand{\leq}{\leqslant}
\renewcommand{\geq}{\geqslant}

\newcommand{\nnegint}{\N}

\newcommand{\posint}{\N^*}

\newcommand{\scalprod}[1]{\langle #1 \rangle}

\newcommand{\rk}{\mathrm{rk}}

\begin{document}

\title[Hessenberg decomposition of matrix fields]{Hessenberg decomposition of matrix fields\\ and bounded operator fields}
\author{Beno\^it Jacob}
\address{University of Toronto\\
Dept. of Mathematics\\
40 St. George Street, room 6290\\
Toronto, Ontario M5S 2E4\\
Canada}
\email{bjacob@math.toronto.edu}

\begin{abstract}
Hessenberg decomposition is the basic tool used in computational linear algebra to approximate the eigenvalues of a matrix. In this article, we generalize Hessenberg decomposition to continuous matrix \emph{fields} over topological spaces. This works in great generality: the space is only required to be normal and to have finite covering dimension. As applications, we derive some new structure results on self-adjoint matrix fields, we establish some eigenvalue separation results, and we generalize to all finite-dimensional normal spaces a classical result on trivial summands of vector bundles. Finally, we develop a variant of Hessenberg decomposition for fields of bounded operators on a separable, infinite-dimensional Hilbert space.
\end{abstract}

\maketitle

\tableofcontents

\section*{Introduction}

A \emph{Hessenberg matrix} is a complex square matrix that is zero outside of the upper triangle and first subdiagonal: in other words, it is a matrix of the form
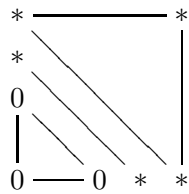
\[
\xymatrix@!@=2pt{
\ast\ar@{-}[rrrrdddd]\ar@{-}[rrrr] & & & & \ast\ar@{-}[dddd] \\
\ast\ar@{-}[rrrddd] & & & & \\
0\ar@{-}[rrdd] & & & & \\
 & & & & \\
0\ar@{-}[rr]\ar@{-}[uu]& & 0 & \ast & \ast
}
\]
Some authors \cite{GvL96} call a Hessenberg matrix \emph{unreduced} if all the coefficients on the first subdiagonal are nonzero. In this article, most of our Hessenberg matrices will happen to have positive (nonzero) coefficients on the first subdiagonal.\\

A \emph{Hessenberg decomposition} of a square matrix $m$ is a decomposition
$$ m = uhu^*$$
where $u$ is unitary and $h$ is Hessenberg. This decomposition is very important to computational linear algebra, because it is the main step toward unitary triangularization, which is decomposing a matrix $m$ as
$$ m = vtv^*$$
where $v$ is unitary and $t$ is triangular. Indeed, for $n\times n$ matrices, the Hessenberg decomposition can be computed exactly in $O(n^3)$ operations, and from there, to any given degree of precision, the unitary triangularization can be computed approximately in just $O(n^2)$ operations, and even just $O(n)$ operations in the self-adjoint case. This is explained in \cite{GvL96}, Chapters 7 and 8. Unitary triangularization is of course very important as it gives the eigenvalues, and amounts in the self-adjoint case to diagonalization.\\

The basic observation of this article is that Hessenberg decomposition can be generalized to work for matrix \emph{fields}, that is, continuously over a topological space, and that that can be used to derive new structure results on matrix fields. Our main Hessenberg decomposition results are summarized in Theorem \ref{hessenbergsummary}:
\begin{theorem*}[See Theorem \ref{hessenbergsummary}]
Let $0\leq d<\infty$. Define $c$ as:
$$c=\left\lbrace\begin{array}{ll}
0 & \text{if $d\leq 1$} \\
2 & \text{if $2\leq d\leq 3$} \\
\left\lceil \frac d2 \right\rceil + 1 & \text{if $d\geq 4$}.
\end{array}\right.$$
Let $X$ be a normal space of covering dimension $d$. Let $n\geq 1$. Let $f\in\continuous(X,\mat_n)$. Let $\epsilon\in\continuous(X,\R_{>0})$. There exist $g\in\continuous(X,\mat_n)$ and $u\in\continuous(X,\unitaries_n)$ such that for all $x\in X$,
\begin{itemize}
\item $g(x)-f(x)$ is self-adjoint,
\item $\Vert f(x)-g(x)\Vert <\epsilon(x)$,
\item the matrix $(ugu^*)(x)$ belongs to $\hess_n^{n-c}$ (see Definition \ref{hess}).
\end{itemize}
\end{theorem*}

We then derive some applications of that theorem.\\

First, Theorem \ref{struc} is a general decomposition result for self-adjoint matrix fields: it says that if we allow an arbitrarily small perturbation, conjugation by a unitary field \emph{and} a rank one perturbation, then any self-adjoint matrix field decomposes as
$$\lambda_1p_1+\cdots+\lambda_kp_k+r$$
where the $p_i$ are mutually orthogonal rank one projections, the $\lambda_i$ are real-valued functions, and $r$ is orthogonal to the $p_i$ and is zero outside of a block of size roughly half the covering dimension of the base space.\\

Corollary \ref{strucdim3} gives a different kind of structure result in the special case of base spaces of covering dimension 3.\\

Then, we establish some eigenvalue separation results: a generic result as Theorem \ref{separationdefault}, followed by specialized theorems in low dimensions: Theorem \ref{separationdim2} shows that over a base space of dimension at most 2, the eigenvalues can be completely separated, generalizing a result of Choi and Elliott \cite[Theorem 1]{CE90}; and Theorem \ref{separationdim4} gives an optimal eigenvalue separation result when the base space has dimension 3 or 4 --- indeed, as we will recall in section \ref{section_separation}, complete separation of eigenvalues is not possible in general as soon as the dimension of the base space is more than 2.\\

We then establish a specialized variant of Hessenberg decomposition for the case of projection fields: that is Theorem \ref{bhessenbergproj}. As an application (Theorem \ref{trivialsummand}), we generalize a classical theorem on vector bundles according to which over a space of dimension $d$, any complex vector bundle of rank $n$ has a trivial summand of rank roughly $n-d/2$. This has long been known for CW-complexes, and the case of compact Hausdorff spaces can be reduced to that case, as noted by Phillips \cite[Proposition 4.2]{P07}. Our Theorem \ref{trivialsummand} shows that this actually works for all normal spaces.\\

In the last section of this article, we replace matrices by bounded operators on an infinite-dimensional, separable Hilbert space. Here, a classical analogue of the notion of being diagonalizable with eigenvalues of multiplicity 1, is the notion of a cyclic operator. An operator is said to be cyclic if there exists a cyclic vector for it. It is well-known (see \cite{H82}, Chapter 18) that cyclic operators do not form a norm-dense subset of bounded operators, and that their complement is norm-dense. Therefore, contrary to the case of finite matrices, here, already in the case where $X$ is a point, it is nontrivial to say anything. Fortunately, there still is a notion of Hessenberg operators (see Definition \ref{hessenbergoperatordef}), that have many interesting properties. They are cyclic, and it is well-known (Lemma \ref{cyclichessenberg}) that cyclic operators are exactly the operators of the form
$$u^*hu$$
with $u$ unitary and $h$ Hessenberg. Moreover, the self-adjoint Hessenberg operators are exactly the Jacobi operators, whose spectral theory and inverse spectral theory are so well-developed (see \cite{T00}). We obtain the following result:
\begin{theorem*}[See Theorem \ref{hessenbergoperators}]
Let $X$ be either a compact space, or a finite-dimensional normal space. Let $H=\ell^2(\posint)$. Let $f$ be a strongly continuous map from $X$ to $B(H)$. For any $\epsilon\in\continuous(X,\R_{>0})$, there exist maps $g$ and $v$ from $X$ to $B(H)$ such that, letting $h=v^*gv$, the following properties are satisfied:
\begin{itemize}
\item $v(x)$ is an isometry for all $x\in X$.
\item $h(x)$ is Hessenberg for all $x\in X$.
\item $v$, $g$, $h$ are strongly continuous.
\item $\Vert f(x)-g(x)\Vert<\epsilon(x)$ for all $x\in X$. Moreover, $f-g$ is norm-continuous, compact, and self-adjoint.
\end{itemize}
\end{theorem*}
Notice that in this theorem, $v(x)$ is only an isometry, not a unitary in general. The statement would be wrong if we said ``unitary'' instead of ``isometry'', again because cyclic operators do not form a norm-dense subset of bounded operators.\\

The author thanks George Elliott and Leonel Robert for helpful conversations, and \'Etienne Blanchard and N. Christopher Phillips for helpful comments on an early version of the present article.\\

\section*{Notations and terminology}

$\nnegint$ is the set of non-negative integers, $\posint$ is the set of positive integers, $\R_{>0}$ is the set of positive real numbers. For $x\in\R$ we let $\lceil x\rceil$ denote the smallest integer $n$ such that $n\geq x$.

$\continuous(X,Y)$ is the set of all continuous maps from a space $X$ to a space $Y$. The notion of topological dimension used throughout this article is the Lebesgue covering dimension, whose definition is recalled in Section \ref{dim}. By the dimension of a space $X$, we mean its covering dimension, which we denote by $\dim X$ (see section \ref{dim}).

$\mat_n$ is the set of complex $n\times n$ matrices. The notation $m_{ij}$ means the $(i,j)$-th coefficient of the matrix $m$, that is the coefficient at row $i$ and column $j$, with the numbering starting at $1$.

$\mat_n^\sa$ is the subset of self-adjoint matrices, $\unitaries_n$ is the subset of unitary matrices. We will introduce a notation $\hess_n^k$ in Definition \ref{hess} and a notation $\bhess_n^k$ in Definition \ref{bhess}.

A matrix is called \emph{positive} if it is of the form $aa^*$ for some matrix $a$. This is what is sometimes called ``positive indefinite''. Thus we consider $0$ a positive matrix, even though we do not consider it a positive number, an inconsistent by usual terminology. A matrix $m$ is negative if $-m$ is positive.

Given two matrices $a\in\mat_n$ and $b\in\mat_p$, we let $a\oplus b$ denote the block-diagonal matrix $\left(\begin{smallmatrix}a & 0 \\ 0 & b\end{smallmatrix}\right)\in\mat_{n+p}$. We let $\mat_n\oplus\mat_p$ denote the subalgebra of $\mat_{n+p}$ consisting of all matrices of that form. By convention, we let $\mat_0=0$ and $\mat_n\oplus\mat_0=\mat_n$.

On a Hilbert space $H$, we only use the norm topology. The scalar product is denoted $\scalprod{\cdot,\cdot}$ and is linear in the first variable. We let $B(H)$ denote the set of bounded operators on $H$. On $B(H)$, by the \emph{strong} and \emph{weak} topologies we mean the operator topologies, \ie the SOT and WOT. The adverbs \emph{strongly} and \emph{weakly} refer to these topologies. For example, to say that a map $f:X\rightarrow B(H)$ is strongly continuous means that for all $\xi\in H$, the map $x\mapsto f(x)\xi$ is continuous. When we apply operator terminology to maps $f:X\rightarrow B(H)$, we mean it pointwise. For example, to say that $f$ is compact means that $f(x)$ is a compact operator for all $x\in X$.

In section \ref{operators} we will fix ourselves $H=\ell^2(\posint)$, we will let $(e_i)_{i\in\posint}$ denote its standard Hilbert basis, and we will let $a_{ij}=\scalprod{ae_i,e_j}$ denote the $ij$-th matrix coefficient of an operator $a\in B(H)$.

\section{Lemmas in topological dimension theory}\label{dim}

Let us first recall the classical notion of Lebesgue covering dimension, often called topological dimension. Given a topological space $X$, a \emph{refinement} of an open covering $(U_i)_{i\in I}$ of $X$ is an open covering $(V_j)_{j\in J}$ of $X$ such that for all $j\in J$, there exists $i\in I$ such that $V_j\subset U_i$.
\begin{definition}[See \cite{N64}]
Let $d\in\nnegint$. A topological space $X$ is said to have \emph{dimension at most $d$} if any open covering of it has a refinement $(V_j)_{j\in J}$ such that for all $x\in X$, the set $\{j\in J,\;x\in V_j\}$ has at most $d+1$ elements.
\end{definition}
Obviously, $\dim X$ is then defined as the smallest $d$ such that $X$ has dimension at most $d$, or $\infty$ if no such $d$ exists. It is true that $\dim\R^d=d$.\\

Let us also recall the notion of a \emph{normal space}, which is the only ``separation axiom'' that we will use throughout this article:
\begin{definition}
A topological space is said to be \emph{normal} if any two disjoint closed subsets have disjoint neighborhoods. In other words: for all closed subsets $F,G$ of $X$, if $F\cap G=\emptyset$ then there exist open subsets $U,V$ of $X$ such that $F\subset U$, $G\subset V$, $U\cap V=\emptyset$.
\end{definition}

Recall the following classical theorem in dimension theory:
\begin{theorem}[See \cite{N64}, Theorem VII.9]
\label{characdimavoid}
Let $X$ be a normal space. Let $n\in\nnegint$. The following are equivalent:
\begin{enumerate}
\item \label{cond_dimXleqd} $\dim X \leq n$.
\item \label{cond_avoidpointincube} For any $f\in\continuous(X,[0;1]^{n+1})$, for any $\epsilon>0$, for any $y\in[0;1]^{n+1}$, there exists $g\in\continuous(X,[0;1]^{n+1})$ such that $\Vert f-g\Vert <\epsilon$ and $y\not\in g(X)$.
\end{enumerate}
\end{theorem}

We will need some variants and refinements of the implication \ref{cond_dimXleqd}$\Rightarrow$\ref{cond_avoidpointincube} in the above theorem. In order to obtain them, we will simply adapt the classical proof of that theorem. The main technical lemmas used in that proof are the following:

\begin{lemma}[Urysohn's Lemma]\label{urysohnlemma}
Let $X$ be a normal space. Let $F,G$ be disjoint closed subsets of $X$. There exists a continuous function $\phi\in\continuous(X,[0;1])$ such that $\phi(F)=0$ and $\phi(G)=1$.
\end{lemma}
Here, by $\phi(F)=\lambda$ we mean that $\phi(x)=\lambda$ for all $x\in F$.

\begin{lemma}[See \cite{N64}, VII.4.B]\label{bigdimthlemma}
Let $n\in\nnegint$. Let $X$ be a normal space such that $\dim X\leq n$. Let $U_1,\ldots,U_{n+1}$ be open subsets of $X$. Let $F_1,\ldots,F_{n+1}$ be closed subsets of $X$. Suppose that $F_i\subset U_i$ for all $i$. It follows that there exist open subsets $V_1,\ldots,V_{n+1}$ and $W_1,\ldots,W_{n+1}$ of $X$ such that
$$F_i\subset V_i\subset \overline V_i\subset W_i\subset U_i \;\;\;\text{for}\;i=1,\ldots,n+1$$
and
$$\bigcap_{i=1}^{n+1}\left(\overline W_i - V_i\right)\;=\;\emptyset.$$
\end{lemma}

Having recalled these classical lemmas, we can now start proving the lemmas that we will need. The proof of the following lemma follows very closely the classical proof of Theorem \ref{characdimavoid}.

\begin{lemma}\label{avoidzero}
Let $n\in\nnegint$. Let $X$ be a normal space such that $\dim X\leq n$. For any $f\in\continuous(X,\R^{n+1})$ and for any $\epsilon\in\continuous(X,\R_{>0})$, there exists $g\in\continuous(X,\R^{n+1})$ such that for all $x\in X$, $\Vert g(x)-f(x)\Vert<\epsilon(x)$ and $g(x)\neq 0$.
\end{lemma}
\begin{proof}
Write $f=(f_1,\ldots,f_{n+1})$ where the $f_i\in\continuous(X,\R)$ are continuous functions. For $1\leq i\leq n+1$, let
\begin{eqnarray*}
F_i & = & \{x\in X;\;f_i(x)\geq\epsilon(x)\}\\
G_i & = & \{x\in X;\;f_i(x)\leq-\epsilon(x)\}\\
\end{eqnarray*}
Since $F_i\subset X-G_i$ for all $i$, it follows from Lemma \ref{bigdimthlemma} that there exist open subsets $V_1,\ldots,V_{n+1}$ and $W_1,\ldots,W_{n+1}$ of $X$ such that
$$F_i\subset V_i\subset \overline V_i\subset W_i\subset X-G_i \;\;\;\text{for}\;i=1,\ldots,n+1$$
and
\begin{equation}\label{emptycap}\bigcap_{i=1}^{n+1}\left(\overline W_i - V_i\right)\;=\;\emptyset.\end{equation}
Since $X$ is normal, by Urysohn's Lemma \ref{urysohnlemma}, for all $i$ there exists a continuous function $\phi_i:X\rightarrow[-1;1]$ such that $\phi_i(\overline V_i)=1$ and $\phi_i(X-W_i)=-1$. Since $F_i\subset V_i$, we have $\phi_i(F_i)=1$. Since $G_i\subset X-W_i$, we have $\phi_i(G_i)=-1$. We may therefore define a continuous function $g_i:X\rightarrow\R$ by letting, for all $x\in X$,
$$g_i(x)=\left\{\begin{array}{ll}f_i(x) & \text{if}\;x\in F_i\cup G_i \\ \epsilon(x)\phi_i(x) & \text{if}\;x\not\in F_i\cup G_i \end{array}\right.$$
We now define $g:X\rightarrow\R^{n+1}$ by letting $g(x)=(g_1(x),\ldots,g_{n+1}(x))$.
It is clear that $\Vert g_i(x) - f_i(x)\Vert\leq 2\epsilon(x)$ for all $x$ and all $i$, and therefore
$$\Vert g(x)-f(x)\Vert \leq 2\sqrt{n+1}\,\epsilon(x)\;\;\;\text{for all}\;x\in X.$$
It remains to show that $g(x)\neq 0$ for all $x\in X$. Suppose that $g(x)=0$ for some $x\in X$. Then $g_i(x)=0$ for $i=1,\ldots,n+1$. It follows that $x\not\in F_i\cup G_i$, so that $g_i(x)=\epsilon(x)\phi_i(x)=0$. Since $\epsilon(x)>0$, it follows that $\alpha_i(x)=0$. This in turn entails that $x\not\in\overline V_i$ and $x\not\in X-W_i$. Therefore, $x\in W_i-\overline V_i$ for all $i=1,\ldots,n+1$, contradicting equation (\ref{emptycap}).
\end{proof}

Here is now a variant where instead of avoiding just one point, we now avoid any finite number of maps pointwise. This is the first, but not the only place where it is useful to have introduced the non-constant $\epsilon$ in Lemma \ref{avoidzero}. This variant will be especially useful when we will establish separation of eigenvalues over low-dimensional spaces (Theorems \ref{separationdim2} and \ref{separationdim4}).

\begin{lemma}\label{avoidkmaps}
Let $n\in\nnegint$. Let $X$ be a normal space such that $\dim X\leq n$. For any $f\in\continuous(X,\R^{n+1})$, for any $k\in\posint$, for any $h_1,\ldots,h_k\in\continuous(X,\R^{n+1})$, and for any $\epsilon\in\continuous(X,\R_{>0})$, there exists $g\in\continuous(X,\R^{n+1})$ such that for all $x\in X$, $\Vert g(x)-f(x)\Vert<\epsilon(x)$ and
$$g(x)\not\in\{h_1(x),\ldots,h_k(x)\}.$$
\end{lemma}
\begin{proof}
Let us work by induction on $k$. The case $k=1$ obviously reduces to Lemma \ref{avoidzero} by just translating by $-h_1$. Let us now suppose the result to hold for a fixed $k$ and let us establish it for $k+1$. Applying the result for that $k$ and for $\epsilon/2$, we obtain a map $\gamma\in\continuous(X,\R^{n+1})$ such that for all $x\in X$, $\Vert \gamma(x)-f(x)\Vert<\epsilon(x)/2$ and $\gamma(x)\not\in\{h_1(x),\ldots,h_k(x)\}.$ Now let
$$\epsilon'(x) = \min\left(\epsilon(x)/2, \min_{i=1,\ldots,k}\Vert \gamma(x) - h_i(x)\Vert\right).$$
Applying the result for $k'=1$, for $h_1'=h_{k+1}$, and for $\epsilon'$, to the map $\gamma$, we obtain a new map $g\in\continuous(X,\R^{n+1})$ such that for all $x\in X$,
$$\Vert g(x)-f(x)\Vert<\epsilon(x)/2 + \epsilon(x)'\leq\epsilon(x)$$
and $g(x)\neq h_{k+1}(x)$, and such that moreover $\Vert g(x)-\gamma(x)\Vert<\Vert g(x)-h_i(x)\Vert$ for $i=1,\ldots,k$, which entails that $g(x)\neq h_i(x)$ for $i=1,\ldots,k$.
\end{proof}

Here is another variant for maps into topological manifolds, which we will have to apply to the spheres $S^n$ in our main Hessenberg reduction process (Theorem \ref{hessenbergdefault}).

\begin{lemma}\label{avoidinmanifold}
Let $n\in\nnegint$. Let $X$ be a normal space such that $\dim X\leq n$. Let $Y$ be a metric topological manifold of dimension $n+1$, with metric denoted by $d$. Suppose that $Y$ has an atlas consisting of bi-Lipschitz charts. Let $Z$ be a discrete subset of $Y$. For any $f\in\continuous(X,Y)$ and for any $\epsilon\in\continuous(X,\R_{>0})$, there exists $g\in\continuous(X,Y)$ such that for all $x\in X$, $d(g(x),f(x))<\epsilon(x)$ and $g(x)\not\in Z$. Moreover, for any neighborhood $W$ of $Z$ in $Y$, $g$ may be chosen so that $g(x)=f(x)$ for all $x\not\in f^{-1}(W)$.
\end{lemma}
\begin{proof}
It is enough to prove that for any $z\in Z$, for any neighborhood $U$ of $z$ in $Y$, there is an open subset $V$ of $U$ such that $V\cap Z=\{z\}$, and a continuous map $g\in\continuous(X,Y)$ such that for all $x\in X$, $d(g(x),f(x))<\epsilon(x)$, $g(x)\not\in Z$, and if $x\not\in f^{-1}(V)$ then $g(x)=f(x)$.\\

So let $z\in Z$ and let $U$ be a neighborhood of $z$ in $Y$. By assumption, there exists an open subset $V$ of $U$ such that $V\cap Z=\{z\}$ and a bi-Lipschitz map $\phi:V\rightarrow\Omega$ for some open subset $\Omega$ of $\R^{n+1}$. Moreover we may choose $\phi$ and $\Omega$ so that $\phi(z)=0$. For any $r>0$, let $B_r$ denote the open ball in $\R^{n+1}$ of radius $r$ centered at 0 in $\R^{n+1}$. Choose $\eta>0$ so that $$B_{3\eta}\subset \Omega.$$
Let $A=f^{-1}(V)$. Let $f'\in\continuous(A,\Omega)$ be defined as $f'=\phi\circ f$. By Lemma \ref{avoidzero}, there exists a map $h'\in\continuous(A, \R^{n+1})$ such that $\Vert h'-f'\Vert<\eta$ and $0\not\in h'(A)$. Now define a map $g'\in\continuous(A,B_{2\eta})$ as follows. Let $x\in A$. If $f'(x)\in B_\eta$, then set $g'(x)=h'(x)$. If $f'(x)\not\in B_{2\eta}$, then set $g'(x)=f'(x)$. Otherwise, we have $\eta\leq \Vert f'(x) \Vert \leq 2\eta$, we let
$$t = \frac{\Vert f'(x) \Vert}{\eta}-1,$$
and set $g'(x) = t\,f'(x)+(1-t)\,h'(x)$.
Define $g\in\continuous(A,V)$ as $g=\phi^{-1}\circ g'$. Notice that $g$ agrees with $f$ outside of $A$, hence $g$ can be extended to all of $X$ by letting $g(x)=f(x)$ for all $x\not\in A$.
\end{proof}

\section{A class of Hessenberg-like matrices}

\begin{definition}
\label{hess}
For $n\in\posint$ and $k\in\{0,\ldots,n\}$, let $\hess_n^k$ be the set of all matrices $m\in\mat_n$ such that
$$m_{ij}=0\;\;\;\text{whenever}\;j\leq k\;\text{and}\;i\geq j+2$$
and
$$m_{ij}\in\R_{>0}\;\;\;\text{whenever}\;j\leq k\;\text{and}\;i=j+1.$$
\end{definition}

Notice that $\hess_n^0=\mat_n$ and that $\hess_n^{n-1}=\hess_n^n$ is the set of all matrices in $\mat_n$ that are Hessenberg and that have positive coefficients on the first subdiagonal. In particular, since the subdiagonal coefficients are nonzero, these are called \emph{unreduced Hessenberg} matrices. More generally, for any $k\in\{0,\ldots,n\}$, the matrices in $\hess_n^k$ are ``unreduced Hessenberg in their $k$ first columns''. They are the matrices of the form
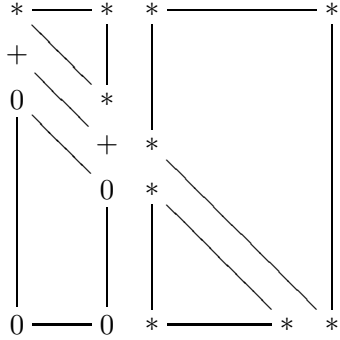
\[
\xymatrix@!@=2pt{
\ast\ar@{-}[rr]\ar@{-}[rrdd] & & \ast\ar@{-}[dd] & \ast\ar@{-}[rrrr]\ar@{-}[ddd] & & & & \ast\ar@{-}[ddddddd] \\
+\ar@{-}[rrdd] & & & & & & & \\
0\ar@{-}[rrdd]\ar@{-}[ddddd] & & \ast & & & & & \\
& & + & \ast\ar@{-}[rrrrdddd] & & & & \\
& & 0\ar@{-}[ddd] & \ast\ar@{-}[ddd]\ar@{-}[rrrddd] & & & & \\
& & & & & & & \\
& & & & & & & \\
0\ar@{-}[rr] & & 0 & \ast\ar@{-}[rrr] & & & \ast & \ast
}
\]
where a ``+'' indicates a coefficient in $\R_{>0}$.\\

Let us make the following observations on the eigenvalues of matrices in $\hess_n^k$.
\begin{proposition}
\label{maxeigmulhess}
Let $n\in\posint$ and $k\in\{0,\ldots,n\}$. Let $p=\max(n-k,1)$. Let $x\in\hess_n^k$. It follows that all eigenvalues of $x$ have multiplicity at most $p$.
\end{proposition}
\begin{proof}
Let $\lambda$ be an eigenvalue of $x$. Notice that $x-\lambda\,1_n$ still belongs to $\hess_n^k$. Let $q = \min(n-1,k)$. It follows from the definition of $\hess_n^k$ that the $q$ first columns of $x-\lambda\,1_n$ are linearly independent, hence $x-\lambda\,1_n$ has rank at least $q$, hence $\dim\ker(x-\lambda\,1_n)\leq n-q=p$, hence $\lambda$ is of multiplicity at most $p$.
\end{proof}
The proof of our second lemma closely follows the proof of a theorem well-known as the Sturm Sequence Property (see \cite{GvL96}, Theorem 8.5.1).
\begin{proposition}
\label{mineigmul1hess}
Let $n\in\posint$ and $k\in\{0,\ldots,n\}$. Let $p=\max(n-k-1,1)$. Let $x\in\hess_n^k$ be self-adjoint. It follows that at most $p$ eigenvalues of $x$ have multiplicity more than 1.
\end{proposition}
\begin{proof}
The statement is trivial if $n\leq 2$, so let us assume that $n\geq 3$. Also the statement is trivial if $k=0$, so let us assume that $k\geq 1$. For $i\in\{0,\ldots,n-1\}$, let $C_i$ denote the bottom-right $(n-i+1)\times (n-i+1)$ corner of $x$. Notice that $c_1=x$ and that $c_n$ is the $1\times1$ matrix $(x_{nn})$. Let
$$p_i(\lambda) = \det(c_i - \lambda\,1_{n-i+1})$$
be its characteristic polynomial. Expanding determinants with respect to the first column, one easily establishes the relation
\begin{equation}
\label{recurdet}
p_i(\lambda) = x_{ii} p_{i+1}(\lambda) - x_{i+1,i}^2 p_{i+2}(\lambda)\;\;\;\text{for}\;1\leq i\leq \min(k,n-2)
\end{equation}
It follows (see \cite{W65}, \S 47) from the minimax principle that if $\lambda_1\geq\ldots\geq\lambda_n$ are the eigenvalues of $x=c_1$ and $\mu_1\geq\ldots\geq\mu_{n-1}$ are the eigenvalues of $c_2$, then the following Interlacing Property holds:
\begin{equation}\label{interlacing}
\lambda_1\geq \mu_1 \geq \lambda_2 \geq \ldots \geq \mu_{n-1} \geq \lambda_n.
\end{equation}
Now let $\lambda$ be an eigenvalue of $x$, and suppose that $\lambda$ has multiplicity more than $1$. It follows from the inequalities (\ref{interlacing}) that $\lambda$ is not only an eigenvalue of $x=c_1$ but also an eigenvalue of $c_2$. In other words:
$$p_1(\lambda)=p_2(\lambda)=0.$$
It follows from equation (\ref{recurdet}) that $x_{2,1}^2 p_{3}(\lambda)=0$. Since $x\in\hess_n^k$ and $k\geq 1$, we have $x_{2,1}\neq 0$ and it follows that
$$p_3(\lambda)=0.$$
We can continue applying equation (\ref{recurdet}) iteratively as long as $i\leq\min(k,n-2)$. So, letting $j=\min(k,n-2)$, in the end we obtain
$$0=p_1(\lambda)=p_2(\lambda)=\cdots=p_{j+2}(\lambda).$$
Thus, $\lambda$ is an eigenvalue of $c_{j+2}$, which is a square matrix of size $n-j-1$ and therefore can't have more than $n-j-1=\max(n-k-1,1)$ distinct eigenvalues.
\end{proof}

\section{Hessenberg reduction of matrix fields}

The goal of this section is to establish the following theorem:

\begin{theorem}\label{hessenbergsummary}
Let $0\leq d<\infty$. Let $c$ be defined as follows:
$$c=\left\lbrace\begin{array}{ll}
0 & \text{if $d\leq 1$} \\
2 & \text{if $2\leq d\leq 3$} \\
\left\lceil \frac d2 \right\rceil + 1 & \text{if $d\geq 4$}.
\end{array}\right.$$
Let $X$ be a normal space of covering dimension $d$. Let $n\geq 1$. Let $f\in\continuous(X,\mat_n)$. Let $\epsilon\in\continuous(X,\R_{>0})$. There exist $g\in\continuous(X,\mat_n)$ and $u\in\continuous(X,\unitaries_n)$ such that for all $x\in X$,
\begin{itemize}
\item $g(x)-f(x)$ is self-adjoint,
\item $\Vert f(x)-g(x)\Vert <\epsilon(x)$,
\item the matrix $(ugu^*)(x)$ belongs to $\hess_n^{n-c}$ (see Definition \ref{hess}).
\end{itemize}
\end{theorem}
We will establish it as three separate propositions: \ref{hessenbergdefault}, \ref{hessenbergdim3}, \ref{hessenbergdim1}.
\begin{proposition}
\label{hessenbergdefault}
Let $n\in\posint$ and $d\in\nnegint$. Let $k=n-\lceil d/2 \rceil-1$. Let $X$ be a normal space such that $\dim X\leq d$. For any $f\in\continuous(X,\mat_n)$ and $\epsilon\in\continuous(X,\R_{>0})$, there exists $g\in\continuous(X,\hess_n^k)$ and $u\in\continuous(X,\unitaries_n)$ such that for all $x\in X$, $(f-u^*gu)(x)$ is self-adjoint and has norm less than $\epsilon(x)$.
\end{proposition}
\begin{proof}
First notice that it is sufficient to prove that for all $p\in\{0,\ldots,k-1\}$, $f\in\continuous(X,\hess_n^p)$, and $\epsilon\in\continuous(X,\R_{>0})$, there exists $g\in\continuous(X,\hess_n^{p+1})$ and $u\in\continuous(X,\unitaries_n)$ such that for all $x\in X$, $(f-u^*gu)(x)$ is self-adjoint and has norm less than $\epsilon(x)$.\\

Let $b$ be the block inside the $(p+1)$-th column of $f$ starting at the $(p+2)$-th row and ending at the $n$-th row. In other words,
$$b=\left(\begin{matrix}b_{p+2,p+1}\\ \vdots \\ b_{n,p+1}\end{matrix}\right).$$

In the following diagram, we represent the matrix $f\in\mat_n(\continuous(X))$ with the block $b$ framed inside it:
\begin{equation}
\label{block_in_general_hessenberg}
\xymatrix@!@=2pt{
\ast\ar@{-}[rr]\ar@{-}[rrdd] & & \ast\ar@{-}[dd] & \ast\ar@{-}[rrrr]\ar@{-}[ddd] & & & & \ast\ar@{-}[ddddddd] \\
+\ar@{-}[rrdd] & & & & & & & \\
0\ar@{-}[rrdd]\ar@{-}[ddddd] & & \ast & & & & & \\
& & + & \ast\ar@{-}[rrrrdddd] & & & & \\
& & 0\ar@{-}[ddd] & \ast\ar@{-}[ddd] & & & & \\
& & & & \ast\ar@{-}[rrdd]\ar@{-}[dd]& & & \\
& & & & & & & \\
0\ar@{-}[rr] & & 0 & \ast & \ast\ar@{-}[rr] & & \ast & \ast
\save
  "5,4"."8,4"*[F]\frm{}
\restore
}
\end{equation}
In this diagram, the $p$ first columns are represented ending with zeros. Of course, this only is an accurate representation if $p\geq 1$. In the case $p=0$, one should simply imagine that the column with the frame is the first column.\\

Notice that since every coefficient $f_{ij}$ is a function in $\continuous(X)$, the block $b$ may be seen as a continuous map
$$b:X\rightarrow\C^{n-p-1}\simeq\R^{2(n-p-1)}.$$
Since $p\leq k-1\leq n-d/2-2$, we have
\begin{equation}\label{compare_d_p} d\leq 2n-2p-4.\end{equation}
By Lemma \ref{avoidzero} and inequality (\ref{compare_d_p}), there exists a continuous map $b':X\rightarrow\R^{2(n-p-1)}$ such that for all $x\in X$, $\Vert b(x)-b'(x)\Vert<\epsilon(x)/n$ and $b'(x)\neq 0$.\\

For $x\in X$, let $r(x)=\Vert b'(x)\Vert\in\R_{>0}$, and let $\beta(x)=b'(x)/r(x)$. Notice that $b'(x)=r(x)\beta(x)$ and that $\beta(x)$ belongs to the unit sphere $S^{2n-2p-3}$ in $\R^{2(n-p-1)}$. \\

By Lemma \ref{avoidinmanifold} and inequality (\ref{compare_d_p}), there exists a continuous map $\beta':X\rightarrow S^{2n-2p-3}$ such that for all $x\in X$,
$$\Vert\beta'(x)-\beta(x)\Vert<\frac{\epsilon(x)}{n\,r(x)}$$
and
$$\beta'(x)\neq\left(\begin{matrix} -1 \\ 0 \\ \vdots \\ 0 \end{matrix}\right)\in\R^{2n-2p-2}.$$
Now let us see $S^{2n-2p-3}$ as the unit sphere in $\C^{n-p-1}$, so that $\beta'$ is a map into $\C^{n-p-1}$. Define a map $b'':X\rightarrow\C^{n-p-1}$ by letting $b''(x)=r(x)\beta'(x)$. Notice that for all $x\in X$, 
$\Vert b''(x)-b(x)\Vert<\epsilon(x)/n$
and, when seeing $b''$ as a map into $\R^{2(n-p-1)}$,
\begin{equation}\label{avoidhalfline}b''(x)\neq\left(\begin{matrix} \lambda \\ 0 \\ \vdots \\ 0 \end{matrix}\right) \;\;\;\text{for all}\;x\in X\;\text{and}\;\lambda\in\R_{\leq0}.\end{equation}
Now let us construct a self-adjoint matrix field $\delta$ such that adding $\delta$ to $f$ has the effect of replacing the block $b$ by $b''$. This is trivially done by letting, for all $x\in X$,
\begingroup
\renewcommand{\arraystretch}{2}
$$\delta_{ij}(x)=\left\{\begin{array}{ll}
b_{i-p-1}''(x) - b_{i-p-1}(x) & \text{if}\;i\geq p+2\;\text{and}\;j=p+1\\
\overline{b_{j-p-1}''(x) - b_{j-p-1}(x)} & \text{if}\;j\geq p+2\;\text{and}\;i=p+1\\
0 & \text{otherwise}
\end{array} \right.$$
\endgroup
Notice that $\Vert\delta(x)\Vert<\epsilon(x)$. Let
$$f''=f+\delta.$$
It remains to construct a unitary field $u\in\continuous(X,\unitaries_n)$ such that $(uf''u^*)(x)$ is always a matrix in $\hess_n^{p+1}$. We will construct $u$ as a \emph{Householder reflection}, a technique widely used in computational linear algebra. Let us look back at the diagram (\ref{block_in_general_hessenberg}) applied to the matrix $f''$. The block framed there is exactly $b''$. This is a map from $X$ to $\C^{n-p-1}$, and we know by (\ref{avoidhalfline}) that it avoids the closed real half-line generated by $(-1,0,\ldots,0)$. Therefore, we may define a continuous map $h:X\rightarrow\C^{n-p-1}$ by letting
$$h(x)\;=\;\frac{b''(x)}{\Vert b''(x)\Vert}+\left(\begin{matrix}1\\0\\ \vdots\\0\end{matrix} \right)\;\;\;\text{for all}\;x\in X,$$
and we then have $h(x)\neq 0$ for all $x\in X$. The vector $h(x)$ is what is known as a \emph{Householder vector}. Let $u_0(x)\in\unitaries_{n-p-1}$ be the orthogonal reflection around the complex line generated by $h(x)$. We have
$$u_0(x)b''(x)\;=\;\left(\begin{matrix} \Vert b''(x)\Vert \\ 0 \\ \vdots \\ 0 \end{matrix}\right)\;\;\;\text{for all}\;x\in X.$$
Moreover, the map $u_0:X\rightarrow \unitaries_{n-p-1}$ thus defined is continuous. Finally we obtain the wanted unitary field $u\in\continuous(X,\unitaries_n)$ by letting, for all $x\in X$,
$$u(x) = \left(\begin{matrix} 1_{p+1} & 0 \\ 0 & u_0(x)\end{matrix}\right).$$
The wanted matrix field $g$ is then obtained as $g=uf''u^*$.
\end{proof}

Now that we have obtained our general Hessenberg decomposition result for spaces of arbitrary finite dimension, let us sharpen it in some special low-dimensional cases. Here is first a refinement for spaces of dimension at most 3. The main new tool that becomes available in dimensions 3 and less, is the possibility to use continuous \emph{Givens rotations} while in the above general case we could only use Householder reflections. In computational linear algebra, Householder reflections and Givens rotations are the two basic types of unitaries used in almost all unitary reduction processes. For further details on these matters, see \cite{GvL96}.

\begin{proposition}
\label{hessenbergdim3}
Let $n\in\posint$. Let $X$ be a normal space such that $\dim X\leq 3$. For any $f\in\continuous(X,\mat_n)$ and $\epsilon\in\continuous(X,\R_{>0})$, there exists $g\in\continuous(X,\hess_n^{n-2})$ and $u\in\continuous(X,\unitaries_n)$ such that for all $x\in X$, $(f-u^*gu)(x)$ is self-adjoint and has norm less than $\epsilon(x)$.
\end{proposition}
\begin{proof}
By Theorem \ref{hessenbergdefault}, we can get all what we need except that $g(x)$ is in $\hess_n^{n-3}$ and we want it to be in $\hess_n^{n-2}$. Since $g\in\continuous(X,\hess_n^{n-3})$, it has the form
\[
\xymatrix@!@=2pt{
\ast\ar@{-}[rr]\ar@{-}[rrdd] & & \ast\ar@{-}[dd] & \ast\ar@{-}[dd] & \ast\ar@{-}[dd] & \ast\ar@{-}[dd] \\
+\ar@{-}[rrdd] & & & & & \\
0\ar@{-}[rrdd]\ar@{-}[ddd] & & \ast & \ast & \ast & \ast\\
& & + & \ast & \ast & \ast\\
& & 0 & \ast & \ast & \ast \\
0\ar@{-}[rr] & & 0 & \ast & \ast & \ast
\save
  "5,4"."6,4"*[F]\frm{}
\restore
}
\]
where ``+'' again denotes a coefficient that is in $\R_{>0}$ for all $x\in X$. In the above diagram we framed the $2\times 1$ block consisting of $g_{n-1,n-2}$ and $g_{n,n-2}$. Call it $b$. Thus $b$ is a continuous map from $X$ to $\C^2$. Since $\dim X\leq 3<\dim\C^2$, it follows from Lemma \ref{avoidzero} that an arbitrarily small perturbation on the matrix field $g$ ensures that $b(x)\neq 0$ for all $x\in X$. This perturbation can be made to be self-adjoint as in the proof of Theorem \ref{hessenbergdefault}: just apply the adjoint perturbation. So we assume without loss of generality that $b(x)\neq 0$ for all $x\in X$. Let $r(x)=\Vert b(x)\Vert$. Define a $2\times 2$ unitary field $u_0$ by letting
$$u_0(x)=\frac{1}{r(x)}\,\left(\begin{matrix} \overline {b_1(x)}  &  \overline {b_2(x)} \\ -b_2(x) & b_1(x) \end{matrix}\right).$$
Notice that for all $x\in X$ we have
$$u_0(x)b(x)=\left(\begin{matrix} r(x) \\ 0 \end{matrix}\right).$$
It follows that if we define a $n\times n$ unitary field $u$, usually called a \emph{Givens rotation}, by letting
$$u(x) = \left(\begin{matrix} 1_{n-2} & 0 \\ 0 & u_0(x)\end{matrix}\right),$$
then the matrix field $ugu^*$ has the wanted form
\[
\xymatrix@!@=2pt{
\ast\ar@{-}[rr]\ar@{-}[rrdd] & & \ast\ar@{-}[dd] & \ast\ar@{-}[dd] & \ast\ar@{-}[dd] & \ast\ar@{-}[dd] \\
+\ar@{-}[rrdd] & & & & & \\
0\ar@{-}[rrdd]\ar@{-}[ddd] & & \ast & \ast & \ast & \ast\\
& & + & \ast & \ast & \ast\\
& & 0 & + & \ast & \ast \\
0\ar@{-}[rr] & & 0 & 0 & \ast & \ast
\save
  "5,4"."6,4"*[F]\frm{}
\restore
}
\]
\end{proof}

Our next Hessenberg decomposition result is for the special case of spaces of dimension at most 1. Here, we can easily obtain the actual \emph{unreduced} Hessenberg form, \ie our matrices are Hessenberg at every point and the subdiagonal coefficients do not vanish.

\begin{proposition}
\label{hessenbergdim1}
Let $n\in\posint$. Let $X$ be a normal space such that $\dim X\leq 1$. For any $f\in\continuous(X,\mat_n)$ and $\epsilon\in\continuous(X,\R_{>0})$, there exists $g\in\continuous(X,\hess_n^n)$ and $u\in\continuous(X,\unitaries_n)$ such that for all $x\in X$, $(f-u^*gu)(x)$ is self-adjoint and has norm less than $\epsilon(x)$.
\end{proposition}
\begin{proof}
By Theorem \ref{hessenbergdim3}, we can get all what we need except that $g(x)$ is in $\hess_n^{n-2}$ and we want it to be in $\hess_n^n$. Since $g\in\continuous(X,\hess_n^{n-2})$, it has the form
\[
\xymatrix@!@=2pt{
\ast\ar@{-}[rr]\ar@{-}[rrdd] & & \ast\ar@{-}[dd] & \ast\ar@{-}[dd] & \ast\ar@{-}[dd] & \ast\ar@{-}[dd] \\
+\ar@{-}[rrdd] & & & & & \\
0\ar@{-}[rrdd]\ar@{-}[ddd] & & \ast & \ast & \ast & \ast\\
& & + & \ast & \ast & \ast\\
& & 0 & + & \ast & \ast \\
0\ar@{-}[rr] & & 0 & 0 & \ast & \ast
\save
  "6,5"."6,5"*[F]\frm{}
\restore
}
\]
where ``+'' again denotes a coefficient that is in $\R_{>0}$ for all $x\in X$. In the above diagram we framed the $1\times 1$ block consisting of $g_{n,n-1}$. All what we need is to make it positive (nonzero). Since $\dim X\leq 1<\dim\C$, it follows from Lemma \ref{avoidzero} that an arbitrarily small perturbation of the matrix field $g$ will ensure that $g_{n,n-1}(x)\neq 0$ for all $x\in X$. As in the proof of Theorem \ref{hessenbergdefault}, this perturbation can be chosen to be self-adjoint just by applying the adjoint perturbation. So we may assume without loss of generality that $g_{n,n-1}(x)\neq 0$ for all $x\in X$. Now let
$$\zeta(x)=\frac{\vert g_{n,n-1}(x)\vert}{g_{n,n-1}(x)}$$
and define a $n\times n$ unitary field $u$ by letting $u(x)=\diag(1,\ldots,1,\zeta(x))$. It is then obvious that $ugu^*$ has the wanted form.
\end{proof}

\section{Application: general structure results for self-adjoint matrix fields}

\begin{theorem}
\label{struc}
Let $d\in\nnegint$. Let $X$ be a normal space of dimension $d$. Let $c=\lceil d/2\rceil+1$. Let $n\in\nnegint$. Assume that $n>c$. Let $k=n-c$. Let $\Pi_n^1$ be the set of rank one (orthogonal) projections in $\mat_n$. Let $f\in\continuous(X,\mat_n^\sa)$. Let $\epsilon\in\continuous(X,\R_{>0})$. It follows that exists $u\in\continuous(X,\unitaries_n)$, $g,r\ldots,p_k\in\continuous(X,\mat_n^\sa)$, $p_1,\ldots,p_k\in\continuous(X,\Pi_n^1)$ and $\lambda_1,\ldots,\lambda_k\in\continuous(X,\R)$ such that the following conditions are satisfied:
\begin{itemize}
\item $\Vert (ufu^*)(x) - g(x) \Vert < \epsilon(x)$ for all $x\in X$;
\item $g = \lambda_1p_1 + \cdots + \lambda_kp_k + r + q$;
\item The projections $p_i$ are mutually orthogonal, and are orthogonal to $r$;
\item The map $r$ is zero outside of the bottom-right $c\times c$ corner: in other words, there exists a map $r_c\in\continuous(X,\mat_c^\sa)$ such that $r=0_{n-c}\oplus r_c$;
\item $q$ may be chosen:
 \begin{itemize}
  \item Either so that for all $x\in X$, $q(x)$ has rank exactly 1 and is positive, and has norm at most $\sqrt2(\Vert f(x)\Vert+\epsilon(x))$;
  \item Or so that for all $x\in X$, $q(x)$ has rank exactly 1, is negative, and has norm at most $\sqrt2(\Vert f(x)\Vert+\epsilon(x))$, and so that if $f$ is positive then $r,p_1,\ldots,p_k$ are positive too;
  \item Or so that for all $x\in X$, $q(x)$ has rank exactly 2, trace $0$, and norm at most $\Vert f(x)\Vert+\epsilon(x)$, and so that if $f$ is positive then $r,p_1,\ldots,p_k$ are positive too;
 \end{itemize}
\item In either case, the spectral projections of $q$ are trivial;
\item $p_1+\cdots+p_k$ is trivial too (actually it is constant equal to $1_k\oplus 0_c$);
\item $\lambda_1(x)>\ldots>\lambda_k(x)$ for all $x\in X$.
\end{itemize}
\end{theorem}
\begin{remark}
Notice that while $r,p_1,\ldots,p_k$ are mutually orthogonal, $q$ is not necessarily orthogonal to any of them. Also notice that when $d\leq 2$, there is the stronger result that the eigenvalues can be completely separated, see Theorem \ref{separationdim2}, which provides a splitting into an orthogonal sum of multiples of rank one projections. Finally, notice that when $d=3$, Theorem \ref{hessenbergdim3} allows to obtain $c=2$ instead of $c=3$, which is a slight improvement.
\end{remark}
\begin{proof}
Apply Theorem \ref{hessenbergdefault} to obtain $u\in\continuous(X,\unitaries_n)$ and $g\in\continuous(X,\mat_n^\sa)$ such that for all $x\in X$, $\Vert (ufu^*)(x) - g(x) \Vert < \epsilon(x)$ and $g(x)\in\hess_n^{n-c}$. It remains to split $g$ as advertised. Since $g(x)$ belongs to $\hess_n^{n-c}$ and is self-adjoint, it has the following form:
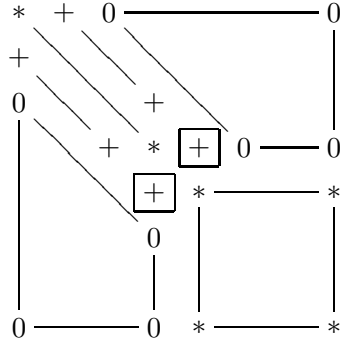
\[
\xymatrix@!@=2pt{
\ast\ar@{-}[rrrddd] & +\ar@{-}[rrdd] & 0\ar@{-}[rrrddd]\ar@{-}[rrrrr] & & & & & 0\ar@{-}[ddd] \\
+\ar@{-}[rrdd] & & & & & & & \\
0\ar@{-}[rrrddd]\ar@{-}[ddddd] & & & + & & & & \\
& & + & \ast & + & 0\ar@{-}[rr] & & 0 \\
& & & + & \ast\ar@{-}[ddd]\ar@{-}[rrr] & & &\ast\ar@{-}[ddd] \\
& & & 0\ar@{-}[dd] & & & & \\
& & & & & & & \\
0\ar@{-}[rrr] & & & 0 & \ast\ar@{-}[rrr] & & & \ast
\save
  "5,4"."5,4"*[F]\frm{}
\restore
\save
  "4,5"."4,5"*[F]\frm{}
\restore
}
\]
The two framed coefficients $g_{k+1,k}(x)$ and $g_{k,k+1}(x)$ are of course equal: call them $\mu(x)$. In the case where we want $q$ of constant rank $2$, let
$$ q = 0_{k-1} \oplus \left(\begin{matrix} 0 & \mu \\ \mu & 0 \end{matrix}\right) \oplus 0_{c-1}.$$
In the case where we want $q$ of constant rank 1 and positive, let
$$ q = 0_{k-1} \oplus \left(\begin{matrix} \mu & \mu \\ \mu & \mu \end{matrix}\right) \oplus 0_{c-1}.$$
In the last case where we want $q$ of constant rank 1 and negative, let
$$ q = 0_{k-1} \oplus \left(\begin{matrix} -\mu & \mu \\ \mu & -\mu \end{matrix}\right) \oplus 0_{c-1}.$$
Either way, $g-q$ has this form:
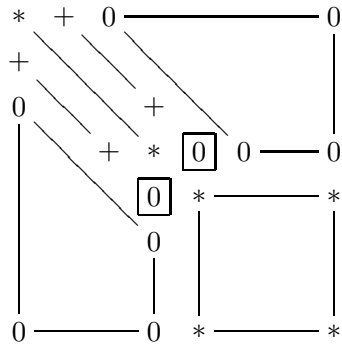
\[
\xymatrix@!@=2pt{
\ast\ar@{-}[rrrddd] & +\ar@{-}[rrdd] & 0\ar@{-}[rrrddd]\ar@{-}[rrrrr] & & & & & 0\ar@{-}[ddd] \\
+\ar@{-}[rrdd] & & & & & & & \\
0\ar@{-}[rrrddd]\ar@{-}[ddddd] & & & + & & & & \\
& & + & \ast & 0 & 0\ar@{-}[rr] & & 0 \\
& & & 0 & \ast\ar@{-}[ddd]\ar@{-}[rrr] & & &\ast\ar@{-}[ddd] \\
& & & 0\ar@{-}[dd] & & & & \\
& & & & & & & \\
0\ar@{-}[rrr] & & & 0 & \ast\ar@{-}[rrr] & & & \ast
\save
  "5,4"."5,4"*[F]\frm{}
\restore
\save
  "4,5"."4,5"*[F]\frm{}
\restore
}
\]
Let now $h$ be the top-left $k\times k$ block and $r$ be the bottom-right $c\times c$ block in $g-q$, as illustrated on this diagram:
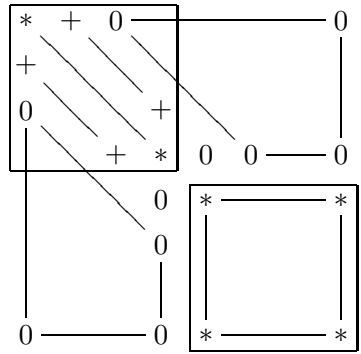
\[
\xymatrix@!@=2pt{
\ast\ar@{-}[rrrddd] & +\ar@{-}[rrdd] & 0\ar@{-}[rrrddd]\ar@{-}[rrrrr] & & & & & 0\ar@{-}[ddd] \\
+\ar@{-}[rrdd] & & & & & & & \\
0\ar@{-}[rrrddd]\ar@{-}[ddddd] & & & + & & & & \\
& & + & \ast & 0 & 0\ar@{-}[rr] & & 0 \\
& & & 0 & \ast\ar@{-}[ddd]\ar@{-}[rrr] & & &\ast\ar@{-}[ddd] \\
& & & 0\ar@{-}[dd] & & & & \\
& & & & & & & \\
0\ar@{-}[rrr] & & & 0 & \ast\ar@{-}[rrr] & & & \ast
\save
  "1,1"."4,4"*[F]\frm{}
\restore
\save
  "5,5"."8,8"*[F]\frm{}
\restore
}
\]
Implicitly we extend $r$ by zero to make it a $n\times n$ matrix. Notice that $h(x)\in\hess_{k}^{k}$ and therefore, by Proposition \ref{maxeigmulhess}, $h(x)$ has $k$ distinct eigenvalues for all $x\in X$. Call them $\lambda_1(x)>\ldots>\lambda_k(x)$ and let $p_i(x)$ be the spectral projection of $h(x)$ corresponding to $\lambda_i(x)$, again implicitly extended by zero to make it a $n\times n$ matrix. This completes the construction, and the remaining conditions are then easily shown to be satisfied.
\end{proof}

Following a different approach, one might ask if other structure results become available when the base space $X$ has low enough dimension. For spaces of dimension up to 2, we will see below that eigenvalues can be completely separated: see Theorem \ref{separationdim2}. This shows that the self-adjoint field then splits as an orthogonal sum of multiples of rank one projections, which is about the strongest structure result that one might hope for. However, already in dimension 3 this isn't possible anymore, as we will discuss in Section \ref{section_separation}. Hence, the following easy corollary of Theorem \ref{hessenbergdim3} is perhaps worth mentioning:
\begin{corollary}
\label{strucdim3}
Let $X$ be a normal space of dimension at most 3. Let $f\in\continuous(X,\mat_n^\sa)$. Let $\epsilon\in\continuous(X,\R_{>0})$. There exists $g\in\continuous(X,\mat_n^\sa)$ and $u\in\continuous(X,\unitaries_n)$ such that we have $\Vert (ufu^*)(x) - g(x)\Vert < \epsilon(x)$ for all $x\in X$, and there exists an open subset $U$ of $X$ such that:
\begin{itemize}
\item For all $x\in U$, $g(x)$ has $n$ distinct eigenvalues
\item For all $x\in X-U$, $g(x)$ has the block-diagonal form
$$g(x)=\left(\begin{matrix}g'(x)& 0 \\ 0 & \lambda(x)\end{matrix}\right)$$
where $g'(x)$ is a $(n-1)\times(n-1)$ has $n-1$ distinct eigenvalues, and $\lambda(x)\in\R$.
\end{itemize}
\end{corollary}
\begin{proof}
Apply Theorem \ref{hessenbergdim3} and construct $U$ as the set of all $x\in X$ such that $g_{n,n-1}(x)\neq 0$. The statements on distinct eigenvalues follow from Proposition \ref{maxeigmulhess}.
\end{proof}

\section{Application: separation of eigenvalues of self-adjoint matrix fields}
\label{section_separation}

By separation of eigenvalues, we mean results showing that, given a matrix field, by applying a small perturbation to it, one can ensure that at every point the matrices do not have eigenvalues of too high multiplicity, and do not have too many eigenvalues of multiplicity more than 1.\\

We do not claim that the approach presented in this section is the best possible, but it is a quick application of our Hessenberg decomposition results, and it does give the strongest possible conclusion at least for self-adjoint fields over base spaces of dimension up to 4. To our knowledge, these results are not to be found in preexisting literature, although the following results should be mentioned. First, Choi and Elliott \cite[Theorem 1]{CE90} give the special case of compact metrizable spaces of dimension up to 2. Second, Phillips \cite[Proof of Theorem 3.3]{P94} gives a different but related result using another method that is probably better and which we believe could also be used to rederive and generalize all the results that we present in this section. That however would require one to first establish some more advanced variants of the dimension-theoretic lemmas in section \ref{dim} of the present article.\\

Here is a generic eigenvalue separation result, as a straightforward corollary of our Hessenberg decomposition. Below we will also prove finer results for spaces of low dimension.
\begin{theorem}
\label{separationdefault}
Let $0\leq d<\infty$. Let $X$ be a normal space such that $\dim X=d$. Let $n\geq 1$. Let $f\in\continuous(X,\mat_n^\sa)$. Let $\epsilon\in\continuous(X,\R_{>0})$. There exists $g\in\continuous(X,\mat_n^\sa)$ such that for all $x\in X$,
\begin{itemize}
\item $\Vert f(x)-g(x)\Vert <\epsilon(x)$,
\item all eigenvalues of $g(x)$ have multiplicity at most $\lceil d/2\rceil+1$,
\item at most $\lceil d/2\rceil$ eigenvalues of $g(x)$ have multiplicity more than 1.
\end{itemize}
\end{theorem}
\begin{remark}
Notice that in particular, this means that $g(x)$ has at least $n-d-1$ eigenvalues of multiplicity $1$. Thus, for large $n$ and fixed $d$, almost all eigenvalues of $g(x)$ have multiplicity 1.
\end{remark}
\begin{proof}
Apply Theorem \ref{hessenbergdefault} to $f$ to obtain a matrix field $g\in\continuous(X,\hess_n^k)$ with $k=n-\lceil d/2\rceil-1$. Since $f$ and $g-f$ are self-adjoint, so is $g$. The wanted properties on the eigenvalues of $g(x)$ are then given by Propositions \ref{maxeigmulhess} and \ref{mineigmul1hess}.
\end{proof}

The next two theorems are finer variants for spaces of low dimension, giving optimal eigenvalue separation results up to dimension 4. Let us first establish \emph{total} separation of eigenvalues when that it possible. But when is that possible? Choi and Elliot \cite[Theorem 1]{CE90} show that that is possible for compact metrizable spaces of dimension up to 2, and provide (essentially) the following counterexample showing that that is not possible already on the Euclidean 3-ball $B^3$:
\begin{eqnarray*}
f\::\:B^3 & \rightarrow & \mat_2(\C)\\
\left(\begin{matrix} x \\ y \\ z \end{matrix}\right) & \mapsto & \frac12 \left(\begin{matrix} 1-z & x+iy \\ x-iy & 1+z\end{matrix}\right).
\end{eqnarray*}
Notice that the restriction of $f$ to the boundary $S^2$ or $B^3$ is the Bott projection, which is a nontrivial rank one projection on $S^2$ (the corresponding line bundle is the Hopf bundle). If $f$ could be approximated by normal matrix fields with distinct eigenvalues, then an approximation of the Bott projection would extend to $B^3$, hence would be trivial, a contradiction.\\

This shows that the following is optimal as far as total separation of eigenvalues is concerned; below we will also give an optimal result for dimensions 3 and 4.
\begin{theorem}
\label{separationdim2}
Let $X$ be a normal space such that $\dim X\leq 2$. Let $n\geq 1$. Let $f\in\continuous(X,\mat_n^\sa)$. Let $\epsilon\in\continuous(X,\R_{>0})$. There exist $g\in\continuous(X,\mat_n^\sa)$ such that for all $x\in X$, $\Vert f(x)-g(x)\Vert <\epsilon(x)$, and $g(x)$ has $n$ distinct eigenvalues.
\end{theorem}
\begin{proof}
The statement is trivial if $n=1$, so we assume without loss of generality that $n\geq 2$. Apply Theorem \ref{hessenbergdim3} to $f$ to obtain a self-adjoint matrix field $g\in\continuous(X,\hess_n^k)$ with $k=n-2$. Since $g\in\continuous(X,\hess_n^{n-2})$ and is self-adjoint, it has the form
\[
\xymatrix@!@=2pt{
\ast\ar@{-}[rrdd] & +\ar@{-}[rrdd] & 0\ar@{-}[rrdd]\ar@{-}[rrr] &  &  & 0\ar@{-}[dd] \\
+\ar@{-}[rrdd] & & & & & \\
0\ar@{-}[rrdd]\ar@{-}[ddd] & & \ast & + & 0 & 0\\
& & + & \ast & + & 0\\
& & 0 & + & \ast & \ast \\
0\ar@{-}[rr] & & 0 & 0 & \ast & \ast
\save
  "6,5"."6,6"*[F]\frm{}
\restore
}
\]
where ``+'' again denotes a coefficient that is in $\R_{>0}$ for all $x\in X$. The framed block, which we will call $b$, is a continuous map from $X$ to $\C\times\R$: indeed, since $g$ is self-adjoint, its diagonal coefficients are real. For $x\in X$, let $\lambda_1(x)>\ldots>\lambda_{n-1}(x)$ denote the eigenvalues of the top-left $(n-1)\times(n-1)$ corner of $g$. Notice that they are all distinct by Proposition \ref{maxeigmulhess}. Since $\dim(\C\times\R)=3$ and $\dim X\leq 2$, we may apply Lemma \ref{avoidkmaps}, showing that an arbitrarily small perturbation ensures that for all $x\in X$,
$$b(x)\not\in\{(0,\lambda_1(x)),\ldots,(0,\lambda_{n-1}(x))\}.$$
Moreover, this can be achieved by a self-adjoint perturbation on $g$, by applying a suitable perturbation on the other side of the diagonal. We then have that for all $x\in X$, either $g_{n,n-1}(x)\neq0$, or $g_{n,n}(x)\neq\lambda_i(x)$ for all $i$. If $g_{n,n-1}(x)\neq 0$, then the proof of Proposition \ref{maxeigmulhess} gives that all eigenvalues of $g(x)$ have multiplicity 1. If on the other hand $g_{n,n-1}(x)=0$ and $g_{n,n}(x)\neq\lambda_i(x)$, then $g(x)$ is a block-diagonal matrix of the form $\alpha\oplus\beta$ where the block $\alpha$ is of size $n-1$ and has $n-1$ distinct eigenvalues $\lambda_1(x),\ldots,\lambda_{n-1}(x)$, and the block $\beta$ is just the $1\times1$ matrix $(g_{n,n}(x))$, and we already know that $g_{n,n}(x)\neq\lambda_i(x)$ for all $i$, so $g(x)$ has $n$ distinct values.
\end{proof}

In dimensions 3 and 4, as we already said, complete separation of eigenvalues can't be hoped for, so the following is optimal:
\begin{theorem}
\label{separationdim4}
Let $X$ be a normal space such that $\dim X\leq 4$. Let $n\geq 1$. Let $f\in\continuous(X,\mat_n^\sa)$. Let $\epsilon\in\continuous(X,\R_{>0})$. There exist $g\in\continuous(X,\mat_n^\sa)$ such that for all $x\in X$, $\Vert f(x)-g(x)\Vert <\epsilon(x)$, and $g(x)$ has at least $n-1$ distinct eigenvalues.
\end{theorem}
\begin{remark}
Notice that in particular, this means that $g(x)$ has no eigenvalue of multiplicity more than 2, and has at most one eigenvalue of multiplicity 2.
\end{remark}
\begin{proof}
The statement is trivial if $n\leq 2$, so we assume without loss of generality that $n\geq 3$. Apply Theorem \ref{hessenbergdefault} to $f$ to obtain a matrix field $h\in\continuous(X,\hess_n^k)$ with $k=n-3$. Since $h\in\continuous(X,\hess_n^{n-3})$, it has the form
\[
\xymatrix@!@=2pt{
\ast\ar@{-}[rrdd] & +\ar@{-}[rrdd] & 0\ar@{-}[rrdd]\ar@{-}[rrr] &  &  & 0\ar@{-}[dd] \\
+\ar@{-}[rrdd] & & & & & \\
0\ar@{-}[rrdd]\ar@{-}[ddd] & & \ast & + & 0 & 0\\
& & + & \ast & \overline a & \overline b\\
& & 0 & a & \ast & \ast \\
0\ar@{-}[rr] & & 0 & b & \ast & \ast
\save
  "1,1"."4,4"*[F]\frm{}
\restore
\save
  "5,5"."6,6"*[F]\frm{}
\restore
}
\]
where ``+'' again denotes a coefficient that is in $\R_{>0}$ for all $x\in X$. Let $\mu_1(x)\geq\mu_2(x)$ denote the eigenvalues of the $2\times2$ bottom-right corner. Let $\lambda_1(x)>\ldots>\lambda_{n-2}(x)$ denote eigenvalues of the $(n-2)\times(n-2)$ top-left corner; notice that the $\lambda_i(x)$ are distinct by Proposition \ref{maxeigmulhess}, as that corner belongs to $\hess_{n-2}^{n-2}$. As illustrated in the above diagram, let $a(x)=h_{n-1,n-2}(x)$ and $b(x)=h_{n,n-2}(x)$. Since $\dim\C^2\times\R=5$ and $\dim X\leq 4$, we may apply Lemma \ref{avoidkmaps}, showing that there exists a map $c=\left(\begin{smallmatrix}c_1\\c_2\\c_3\end{smallmatrix}\right):X\rightarrow\C\times\C\times\R$ such that for all $x\in X$, $\Vert c(x)\Vert<\epsilon(x)$ and
\begin{eqnarray}
\label{c_not_equal}
c(x)& \not\in & \left\{
\left(\begin{matrix}-a(x)\\-b(x)\\(\lambda_1-\mu_1)(x)\end{matrix}\right)
,\ldots,
\left(\begin{matrix}-a(x)\\-b(x)\\(\lambda_{n-2}-\mu_1)(x)\end{matrix}\right)
,\right.\\
& & \left.\left(\begin{matrix}-a(x)\\-b(x)\\(\lambda_1-\mu_2)(x)\end{matrix}\right)
,\ldots,
\left(\begin{matrix}-a(x)\\-b(x)\\(\lambda_{n-2}-\mu_2)(x)\end{matrix}\right)
\right\}. \nonumber
\end{eqnarray}
Now use $c$ to define a perturbation $g$ of $h$ as follows. For any $i,j$ let $E_{i,j}$ denote the matrix whose $(i,j)$-th entry is 1 and whose other entries are all 0. For $x\in X$, let
$$\begin{array}{rccclcl}
g(x)& = & h(x) & + & c_1(x)E_{n-1,n-2} &+& \overline{c_1(x)}E_{n-2,n-1} \\
    &   &      & + & c_2(x)E_{n,n-2} &+& \overline{c_2(x)}E_{n-2,n} \\
    &   &      & + & c_3(x)(E_{n-1,n-1} +E_{n,n}). & &
\end{array}$$
Notice that $g$ is still continuous and self-adjoint. Now let us fix ourselves some $x\in X$ for the remainder of this proof. We want to prove that $g(x)$ has at least $n-1$ distinct eigenvalues. For brevity, let
$$a'=g_{n-1,n-2}(x)=a(x)+c_1(x)$$
and
$$b'=g_{n,n-2}(x)=b(x)+c_2(x).$$
We have either $a'\neq0$, or $b'\neq0$, or $a'=b'=0$. The rest of the proof splits into two cases:\\

\emph{First case:} $a'\neq0$ or $b'\neq0$. Let $r=\left\Vert\begin{smallmatrix}a'\\b'\end{smallmatrix}\right\Vert$. Let $u_0$ be the following rotation:
$$u_0=\frac1r\,\left(\begin{matrix}\overline a' & \overline b' \\ -b' & a' \end{matrix}\right).$$
Let $u=1_{n-2}\oplus u_0$. The unitary $u\in\unitaries_n$ is again an example of a \emph{Givens rotation}. The point is that
$$u_0\left(\begin{matrix}a'\\b'\end{matrix}\right)=\left(\begin{matrix}r\\0\end{matrix}\right)$$
so that the matrix $m=ug(x)u^*$ belongs to $\hess_n^{n-2}$. It is thus enough to prove that for any matrix $m\in\hess_n^{n-2}$, $m$ has at least $n-1$ distinct eigenvalues. There are two cases: either $m_{n,n-1}=0$ or not. If $m_{n,n-1}=0$ then the matrix $m$ is block-diagonal with a block in $\hess_{n-1}^{n-1}$ followed by a $1\times 1$ block, so the result follows from Proposition \ref{maxeigmulhess}. If, on the other hand, $m_{n,n-1}\neq 0$, then conjugating by an obvious unitary ensures that $m_{n,n-1}\in\R_{>0}$ and therefore the matrix $m$ is in $\hess_{n}^{n}$ and the result follows again from Proposition \ref{maxeigmulhess}.\\

\emph{Second case:} $a'=b'=0$. It then follows from the definition of $g$ that $g(x)$ is block-diagonal with a $(n-2)\times(n-2)$ block with the $n-2$ distinct eigenvalues $\lambda_1(x)>\ldots>\lambda_{n-2}(x)$, followed by a $2\times2$ block. It then follows from (\ref{c_not_equal}) and the definition of $g$ that the eigenvalues of the latter $2\times2$ block are not equal to $\lambda_i(x)$ for any $i$, which establishes that $g(x)$ has at least $n-1$ distinct eigenvalues.
\end{proof}

\section{The case of projections}

Projections are a special case of self-adjoint elements and so our Hessenberg decomposition results can readily be applied to them. There is, however, more that we can say in this special case, and it turns out that the following Hessenberg-like form is more useful for projections than the form $\hess_n^k$ considered so far.
\begin{definition}
\label{bhess}
For $n\in\posint$ and $k\in\{1,\ldots,n\}$, let $\bhess_n^k$ denote the set of all matrices $p$ in $\mat_n$ satisfying the following conditions:
\begin{itemize}
\item $p$ is a projection ($p=p^*=p^2$).
\item $p$ is a block diagonal matrix of the following form:
$$p\in \mat_{\alpha_1}\oplus\cdots\oplus\mat_{\alpha_r}\oplus\mat_\beta$$
where $r\in\nnegint$, where $\alpha_i\in\{1,2\}$ for all $i$, and where $\beta$ is as follows:
$$\beta=\left\lbrace\begin{array}{ll}
n-k & \text{if $\alpha_r=1$} \\
n-k-1 & \text{if $\alpha_r=2$}.
\end{array}\right.$$
Notice that it is implicitly required that $\alpha_1+\cdots+\alpha_r+\beta=n$.
\item Outside of the last block of size $\beta$, all matrix coefficients of $p$ are real nonnegative. In other words: $p_{ij}\in\R_{\geq 0}$ whenever $i,j\leq n-\beta$.
\end{itemize}
\end{definition}

\begin{theorem}
\label{bhessenbergproj}
Let $0\leq d<\infty$. Define $c$ as follows:
$$c=\left\lbrace\begin{array}{ll}
0 & \text{if $d\leq 1$} \\
2 & \text{if $2\leq d\leq 3$} \\
\left\lceil \frac d2 \right\rceil + 1 & \text{if $d\geq 4$}.
\end{array}\right.$$
Let $X$ be a normal space of covering dimension $d$. Let $n\geq 1$. Let $p\in\continuous(X,\mat_n)$ be a projection. There exists a projection $q\in\continuous(X,\mat_n)$ such that:
\begin{itemize}
\item $q(x)$ belongs to $\bhess_n^{n-c}$ for all $x\in X$
\item There exists $u\in\continuous(X,\unitaries_n)$ such that $q=upu^*$.
\end{itemize}
\end{theorem}
\begin{proof}
Let $\epsilon$ be any number such that $$0<\epsilon<\frac{1}{24^2n^3}.$$
Apply Theorem \ref{hessenbergsummary} to $p$ and $\epsilon$, seeing $\epsilon$ as a constant function on $X$. Let $g,u$ be as given by the theorem. Let $h=upu^*$. For all $x\in X$, $h(x)$ is a projection and is within distance $\epsilon$ of $g(x)$ which belongs to $\hess_n^{n-c}$. Let $h'(x)$ be the matrix defined by $h'(x)_{ij}=h(x)_{ij}$ whenever $\vert i-j\vert\leq 1$ and $i,j\leq n-c$, and $h'(x)_{ij}=0$ otherwise. Thus $h'$ has this shape:
\[
\xymatrix@!@=2pt{
\ast\ar@{-}[rrrddd] & +\ar@{-}[rrdd] & 0\ar@{-}[rrdd]\ar@{-}[rrrr] & & & & 0\ar@{-}[dd] \\
+\ar@{-}[rrdd] & & & & & & \\
0\ar@{-}[rrdd]\ar@{-}[dddd] & & & + & 0\ar@{-}[rr] & & 0 \\
& & + & \ast\ar@{-}[rrr]\ar@{-}[ddd] & & & \ast\ar@{-}[ddd] \\
& & 0\ar@{-}[dd] & & & & \\
& & & & & & \\
0\ar@{-}[rr] & & 0 & \ast\ar@{-}[rrr] & & & \ast\\
%\save
%  "5,4"."8,4"*[F]\frm{}
%\restore
}
\]
where a ``+'' denotes a positive coefficient. Notice that $h'(x)$ is still within distance $n\epsilon$ of the projection $h(x)$. It follows that $h'(x)^2$ is within distance $4n\epsilon$ of $h'(x)$. But, for all $1\leq j\leq \min(n-2,n-c)$, the $(j+2,j)$-th coefficient of $h'(x)^2$ is equal to $h'(x)_{j+1,j}h'(x)_{j+2,j+1}$. It follows that
$$ h'(x)_{j+1,j}h'(x)_{j+2,j+1} < 4n\epsilon \;\;\;\text{for all}\;x\in X\;\text{and}\;1\leq j\leq \min(n-2,n-c).$$
Therefore, for all such $x$ and $j$, either $h'(x)_{j+1,j}$ or $h'(x)_{j+2,j+1}$ must be smaller than $2\sqrt{n\epsilon}$. This says that for every $x$, the top-left $(n-c)\times(n-c)$ corner of $h'(x)$ is ``nearly'' block-diagonal with blocks of size 1 or 2. It should be noted that this block-diagonal form (\ie the sequence of block sizes) depends on $x$. A similar argument in the $(n-c)$-th column shows that for all $x\in X$, for all $k\geq n-c+2$, either $h'(x)_{n-c+1,n-c}$ or $\vert h'(x)_{k,n-c+1}\vert$ must be smaller than $2\sqrt{n\epsilon}$. Let $\alpha$ be the function defined on $\C$ by letting $\alpha(0)=0$ and
$$\alpha(z)=z\,\frac{\max\left(0,\vert z\vert-2\sqrt{n\epsilon}\right)}{\vert z\vert} \;\;\;\text{for all }\;z\in \C-\{0\}.$$
Define a new matrix field $h''$ by letting
$$h''(x)_{ij}=\alpha(h'(x)_{ij})\;\;\;\text{for all}\;x\in X\;\text{and all}\;i,j.$$
Notice that $h''(x)$ is within distance $2n^{3/2}\epsilon^{1/2}$ of $h'(x)$, hence
$$\Vert h''(x) - h(x)\Vert < 6n^{3/2}\epsilon^{1/2}.$$
Since $\epsilon<1/(24^2n^3)$, it follows that $\Vert h''(x) - h(x)\Vert < 1/4$. Also notice that for every $x\in X$, $h''(x)$ has the block-diagonal shape described in Definition \ref{bhess} for $k=n-c$. Again, it should be noted that this block-diagonal form depends on $x$. Anyway, for every separate $x$, this block diagonal form is preserved when we replace $h''(x)$ by a function of it, by functional calculus. Since $\Vert h''(x) - h(x)\Vert < 1/4$ and $h(x)$ is a projection, we know that $1/2$ does not belong to the spectrum of $h''(x)$, hence if we let $\chi_{(1/2;\infty)}$ be the characteristic function of $(1/2;\infty)$, and let $q(x)=\chi_{(1/2;\infty)}(h''(x))$ by functional calculus, it follows that $q\in\continuous(X,\mat_n)$ is a projection, that $q(x)$ has the same block diagonal shape as $h''(x)$ for each $x\in X$, and also that $\Vert q(x) - h''(x)\Vert \leq \Vert h''(x) - h(x)\Vert < 1/4$ and hence
\begin{equation}\label{distqh}\Vert q(x) - h(x) \Vert < \frac12\cdot\end{equation}
The following argument is then most classical: let
$$z=qh + (1-q)(1-h),$$
notice that it follows from inequality (\ref{distqh}) that $z(x)$ is invertible for all $x\in X$; and construct the unitary $v(x)$ of its polar decomposition as follows:
$$v(x) = z(x) \left(z(x)^*z(x)\right)^{-1/2}.$$
This unitary realizes the unitary equivalence between $q$ and $h$, completing the proof.
\end{proof}

\section{Application: trivial summands of vector bundles}

It is a well-known principle that over a $d$-dimensional space, under certain conditions, complex vector bundles of rank $n$ have trivial summands of rank roughly $n-d/2$. As we said in the introduction, this has long been known for CW-complexes, and the case of compact Hausdorff spaces can be reduced to that case using some advanced results of dimension theory (see \cite[Proposition 4.2]{P07}).\\

In this section, using a completely different approach, we show that this actually works for all normal spaces. This is obtained as a corollary of our Hessenberg decomposition results. Thus, this is a generalization of previously known results, with a far more elementary proof. By a \emph{vector bundle}, we mean a locally trivial complex vector bundle. We will make use of the notion of \emph{finite type} developed by Vaserstein \cite{V86} which, in the case of normal spaces, just means that there exists a finite open covering consisting of open subsets on each of which the bundle is trivial.

\begin{theorem}
\label{trivialsummand}
Let $0\leq d<\infty$. Let $\gamma$ be defined as follows:
$$\gamma=\left\lbrace\begin{array}{ll}
0 & \text{if $d\leq 1$} \\
1 & \text{if $2\leq d\leq 3$} \\
\left\lceil \frac d2 \right\rceil & \text{if $d\geq 4$}.
\end{array}\right.$$
Let $X$ be a normal space such that $\dim X = d$. Let $\xi$ be a vector bundle of finite type over $X$. Let
$$b = \min_{x\in X} \dim \xi_x.$$
Suppose that $b\geq \gamma+1$. It follows that $\xi$ has a trivial summand of rank $b-\gamma$. In other words, $\xi$ possesses cross-sections $s_1,\ldots,s_{b-\gamma}$ such that for all $x\in X$, the family $(s_1(x),\ldots,s_{b-\gamma}(x))$ is linearly independent.
\end{theorem}
\begin{proof}
Since $\xi$ has finite type, by the theorems in \cite{V86}, $\xi$ is isomorphic to the column-space of a projection field $p\in\continuous(X,\mat_n)$ for some $n\in\posint$. By induction on $b$, it is enough to prove that if $b\geq \gamma+1$ then there exists a non-vanishing cross-section. So let us suppose that $b\geq \gamma+1$, and let us construct that cross-section. By Theorem \ref{bhessenbergproj}, we may choose $p$ so that
$$p(x)\in\bhess_n^{n-c}\;\;\;\text{for all}\;x\in X,\;\text{with}\;c=\gamma+1.$$
Since $b\geq c$, for all $x\in X$, we have $\rk(p(x))\geq c$. For $1\leq i\leq n$ and $x\in X$, let $v_i(x)$ denote the $i$-th column of $p(x)$, seen as a vector in $\C^n$. Let $(e_1,\ldots,e_n)$ be the standard basis of $\C^n$. Let us show that:
\begin{equation}
\label{ifcolszero}\text{For all $x\in X$, if $\Vert v_i(x)\Vert< 1/\sqrt2$ for all $i\leq n-c$, then $v_{n-c+1}(x)=e_{n-c+1}$.}
\end{equation}
Indeed, for each $x\in X$, since $p(x)$ belongs to $\bhess_n^{n-c}$, it decomposes as $\pi\oplus \rho$ with $\rho$ either in $\mat_c$ or in $\mat_{c-1}$. Since $p(x)$ is a projection, both $\pi$ and $\rho$ are. Since $\pi$ is a projection, is block-diagonal with blocks of size at most 2, and all the columns of $\pi$ have norm strictly less than $1/\sqrt2$, it follows that $\pi=0$. So $p(x)=0\oplus\rho$. Since $p(x)$ has rank at least $c$, this excludes the case $\rho\in\mat_{c-1}$ and we therefore have $\rho\in\mat_c$. Moreover $\rho$, which is a projection of rank at least $c$, must then be the $c\times c$ identity matrix. We therefore have $p(x)=0_{n-c}\oplus 1_c$. This completes the proof of (\ref{ifcolszero}).\\

For $x\in X$, let $i_x$ denote the smallest $i$ such that $\Vert v_i(x)\Vert\geq 1/\sqrt2$, and let
$$v(x)=v_{i_x}(x).$$
Note that $v$ is clearly a set-theoretic (\ie perhaps not continuous) nonvanishing cross-section of $\xi$. It remains to check that $v$ really is continuous. Let $x_0\in X$. We must show that $v$ is continuous at $x_0$.\\

By (\ref{ifcolszero}), we know that $i_{x_0} \leq n-c+1$. Moreover, if $i_{x_0}=n-c+1$ then (\ref{ifcolszero}) makes it clear that the mapping $x\mapsto i_x$ is constant on a neighborhood of $x_0$, and hence $v$ is continuous at $x_0$ in that case. Also, if $\Vert v(x)\Vert> 1/\sqrt2$, the same conclusion obviously holds.\\

Thus, there only remains to handle the case when $i_{x_0}\leq n-c$ and $\Vert v(x_0)\Vert = 1/\sqrt2$. Since $i_{x_0}\leq n-c$, the $i_{x_0}$-th column of $p(x)$ falls in the area where $p(x)$ is block diagonal with blocks of size 1 or 2. Since $p(x)$ is a projection and the $i_{x_0}$-th column has norm $1/\sqrt2$, we must be in presence of a diagonal block of size 2. Let $j$ be the index such that this $2\times 2$ block is located in columns $j$ and $j+1$. Thus we have
$$i_{x_0}\in\{j,j+1\}$$
and the $2\times 2$ block in question is
$$\left(\begin{matrix} p(x)_{j,j} & p(x)_{j,j+1} \\ p(x)_{j+1,j} & p(x)_{j+1,j+1} \end{matrix}\right).$$
This $2\times 2$ matrix is a projection; since it has a column of norm $1/\sqrt 2$, it must be of rank 1. Since moreover (see Definition \ref{bhess}) its matrix coefficients are real nonnegative, it must be of the form
$$\left(\begin{matrix} \alpha & \sqrt{\alpha-\alpha^2} \\ \sqrt{\alpha-\alpha^2} & 1-\alpha\end{matrix}\right)$$
with $\alpha\in[0;1]$. The norms of the two columns of such a $2\times 2$ matrix are
\begin{equation}\label{norms}\sqrt{\alpha}\;\;\;\text{and}\;\;\;\sqrt{1-\alpha}\end{equation}
respectively. Notice that for any value of $\alpha$, at least one of the above two norms is $\geq 1/\sqrt2$. This shows that there exists a neighborhood $U$ of $x_0$ such that $i_x\leq j+1$ for all $x\in U$. On the other hand, as we already observed, the proof of (\ref{ifcolszero}) makes it clear that $v_i(x)=0$ for all $i<j$, and therefore the neighborhood $U$ may be taken so that
$$i_x\in\{j,j+1\}\;\;\;\text{for all}\;x\in U.$$
Finally, since $\Vert v(x_0)\Vert = 1/\sqrt2$, we know that one of the two numbers in (\ref{norms}) must be equal to $1/\sqrt2$, and it follows that $\alpha=1/2$ and therefore the $2\times 2$ diagonal block in question is equal to
$$\left(\begin{matrix} 1/2 & 1/2 \\ 1/2 & 1/2 \end{matrix}\right).$$
Thus the two columns are equal, which means that it doesn't matter if $i_x$ jumps between $j$ and $j+1$, showing that $v$ is continuous at $x_0$.
\end{proof}

\section{Bounded operator fields}\label{operators}

In our investigation of Hessenberg reduction for matrix fields, our main problem has been that in the last few columns, as we approached the bottom-right of the matrix, we didn't have enough room anymore to continue the process, so we had to stop. This premature ending of the Hessenberg reduction process has been the main limiting factor to the strength of the results that we subsequently obtained, and one may wonder about ways to work around it. Of course, the premature stopping can't be completely avoided in general, because if we could obtain a complete Hessenberg form $H_n^n$, that would in particular imply that we can completely separate eigenvalues, something that we already said is impossible in general over spaces of dimension 3 and up.\\

Another way of working around this problem is to add zeros to the right and to the bottom of our matrix field, so as to make more room. Going one step further one may add infinitely many rows and columns of zeros, making our matrix field a compact operator field; going one more step further, one may start right away with any bounded operator field.\\

Throughout this section, let $H=\ell^2(\posint)$. Let $B(H)$ denote the set of bounded operators on $H$. Let $(e_i)_{i\in\posint}$ be its standard Hilbert basis. For $x\in B(H)$ and $i,j\in \posint$, let
$$x_{ij} = (xe_j, e_i)\in\C$$
be the $(ij)$-th matrix entry of $x$.

\begin{definition}
\label{hessenbergoperatordef}
A bounded operator $x\in B(H)$ is a \emph{Hessenberg operator} if $x_{ij}=0$ whenever $i>j+1$ and $x_{ij}\in\R_{>0}$ whenever $i=j+1$.
\end{definition}

\begin{definition}
A \emph{Jacobi operator} is a self-adjoint Hessenberg operator.
\end{definition}

\begin{definition}
Let $H$ be a Hilbert space. A bounded operator $x\in B(H)$ is a \emph{cyclic operator} if there exists a vector $\xi\in H$ such that the linear span of the family $(x^k\xi)_{k\in\N}$ is dense in $H$. The vector $\xi$ is then called a \emph{cyclic vector} of $x$.
\end{definition}

It is well-known (see \cite{H82}, Chapter 18) that the set of cyclic operators on $H$ is not norm-dense in $B(H)$, and that its complement is norm-dense.\\

The following result is classical:
\begin{lemma}
\label{cyclichessenberg}
A bounded operator $x\in B(H)$ is cyclic if, and only if it is of the form $uhu^*$ with $u$ a unitary operator and $h$ a Hessenberg operator.
\end{lemma}
\begin{proof}
If $h$ is Hessenberg then $e_1$ is a cyclic vector of it, hence $ue_1$ is a cyclic vector of $uhu^*$. Conversely, suppose that $x$ is a cyclic operator. Let $\xi$ be a cyclic vector of $x$. For $k\in \posint$, let $f_k=x^{k-1}\xi$. Then $(f_k)_{k\in \posint}$ is a linearly independent family whose span is dense in $H$. Let $(g^k)_{k\in \posint}$ be the Hilbert basis of $H$ obtained by applying the Gram-Schmidt orthonormalization process to $(f_k)_{k\in \posint}$. Let $u\in B(H)$ be the unitary defined by letting $u(e_k)=g^k$. It is then straightforward to check that $u^*xu$ is Hessenberg.
\end{proof}

Notice that it is a corollary of the previous lemma that any eigenspace of a cyclic operator has dimension 1. The converse is true for diagonalizable finite matrices. These facts suggest to think of cyclicity as a generalized notion of ``multiplicity 1''.\\

In this section, we prove the following theorem:

\begin{theorem}
\label{hessenbergoperators}
Let $X$ be either a compact space or a finite-dimensional normal space. Let $f$ be a strongly continuous map from $X$ to $B(H)$. For any $\epsilon\in\continuous(X,\R_{>0})$, there exist maps $g$ and $v$ from $X$ to $B(H)$ such that, letting $h=v^*gv$, the following conditions hold:
\begin{itemize}
\item $v(x)$ is an isometry for all $x\in X$.
\item $h(x)$ is Hessenberg for all $x\in X$.
\item $v$, $g$, $h$ are strongly continuous.
\item $\Vert f(x)-g(x)\Vert<\epsilon(x)$ for all $x\in X$. Moreover, $f-g$ is norm-continuous, compact, and self-adjoint.
\end{itemize}
\end{theorem}

We first need a few lemmas:

\begin{lemma}
\label{avoidinfdim}
Let $X$ be either a compact space or a finite-dimensional normal space. Let $f\in\continuous(X,H)$ and $\epsilon\in\continuous(X,\R_{>0})$. It follows that there exists $g\in\continuous(X,H)$ such that for all $x\in X$, $\Vert f(x)-g(x)\Vert<\epsilon(x)$ and $g(x)\neq 0$.
\end{lemma}
\begin{proof}
The case when $X$ is normal and finite-dimensional reduces to Lemma \ref{avoidzero}. Let us handle the case when $X$ is compact. For $k\in\posint$ let $U_k$ be the set of all $x\in X$ such that
\begin{equation}
\label{inequk}
\vert\scalprod{f(x),e_i}\vert<\epsilon(x)/2\;\;\;\text{for all}\;i\geq k.
\end{equation}
The sequence $(U_k)_{k\in\posint}$ is an open covering of $X$, and is increasing in the sense of inclusion. Since $X$ is compact, it follows that there exists some $k\in\posint$ such that $U_k=X$, so that inequality (\ref{inequk}) holds for all $x\in X$. We may therefore obtain the wanted map $g$ just by letting, for all $x\in X$,
$$g(x)=f(x)+\epsilon(x)e_k.\qedhere$$
\end{proof}

The main technical step toward Theorem \ref{hessenbergoperators} is the following lemma:

\begin{lemma}
\label{finitehessenberg}
Let $X$ be either a compact space or a finite-dimensional normal space. Let $f$ be a strongly continuous map from $X$ to $B(H)$. For any $\epsilon\in\continuous(X,\R_{>0})$, there exist sequences $(g^k)_{k\in \posint}$ and $(u^k)_{k\in \posint}$ of maps from $X$ to $B(H)$ such that, letting $h^k=u^kg^ku^{k*}$, the following properties are satisfied:
\begin{itemize}
\item for all $k\in \posint$, $g^k$ and $h^k$ are strongly continuous maps from $X$ to $B(H)$
\item for all $k\in \posint$, $u^k$ is a norm-continuous map from $X$ to $U(H)$.
\item for all $k\in \posint$, $g^{k+1}-g^k$ is norm-continuous, self-adjoint, compact, and $\Vert g^{k+1}(x)-g^k(x)\Vert<\epsilon(x)/2^k$.
\item for all $k,l\in \posint$, if $k\leq l$ then the $k$ first rows of $u^l$ and $u^k$ agree. In other words, $\forall i,j,k,l\in \posint$, if $i\leq k\leq l$ then $u^l_{i,j}=u^k_{i,j}$.
\item for all $k,l\in \posint$, if $1<k\leq l$ then the $k-1$ first columns of $h^k$ and $h^l$ agree, and are in Hessenberg form. In other words:
\begin{itemize}
\item $\forall i,j,k,l\in \posint$, if $j< k\leq l$ then $h^l_{i,j}=h^k_{i,j}$.
\item $\forall i,j,k\in \posint$, if $j< k$ and $i\geq j+2$ then $h^k_{i,j}= 0$.
\item $\forall i,j,k\in \posint$, if $j\leq k$, $j<n$ and $i=j+1$, then $h^k_{i,j}\neq 0$.
\end{itemize}
\end{itemize}
\end{lemma}
\begin{proof}
We construct the sequences $(g^k)_{k\in \posint}$ and $(u^k)_{k\in \posint}$ by induction. We set $g^1=f$, and we let $u^1$ be the constant map, $u^1(x)=\id$ for all $x\in X$. Now suppose that for some $k\in \posint$, the sequences $u^1,\ldots,u^k$ and $g^1,\ldots,g^k$ have been constructed, and let us then construct $g^{k+1}$ and $u^{k+1}$.\\

Since the $k-1$ first columns of $h^k$ are in Hessenberg form, that means that $h^k$ has the following form:
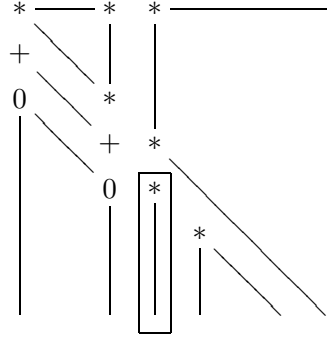
\[
\xymatrix@!@=2pt{
\ast\ar@{-}[rr]\ar@{-}[rrdd] & & \ast\ar@{-}[dd] & \ast\ar@{-}[rrrr]\ar@{-}[ddd] & & & & \\
+\ar@{-}[rrdd] & & & & & & & \\
0\ar@{-}[rrdd]\ar@{-}[ddddd] & & \ast & & & & & \\
& & + & \ast\ar@{-}[rrrrdddd] & & & & \\
& & 0\ar@{-}[ddd] & \ast\ar@{-}[ddd] & & & & \\
& & & & \ast\ar@{-}[rrdd]\ar@{-}[dd]& & & \\
& & & & & & & \\
& & & & & & &
\save
  "5,4"."8,4"*[F]\frm{}
\restore
}
\]
In this diagram we represented the $k-1$ first columns ending with zeros. Of course this only is an accurate representation if $k>1$. In the $k=1$ case, one should simply imagine that the framed column is the first column.\\

Let us do the same kind of ``Householder reflection'' argument that we already made in the proof of Theorem \ref{hessenbergdefault}. In the above diagram, we framed the infinite block in the $k$-th column starting at row $k+1$. Call it $b$:
$$b(x)=\sum_{i=1}^\infty h_{k+i}^k(x)e_i.$$
Thus $b$ is a map from $X$ to $H$. Moreover, since $h^k$ is strongly continuous, $b$ is continuous. Now let $b'$ be the map obtained by taking off the first coefficient in $b$. In other words:
\begin{equation}\label{bprime}b'(x)=\sum_{i=1}^\infty h_{k+i+1}^k(x)e_i.\end{equation}
By Lemma \ref{avoidinfdim}, a small perturbation ensures that $b'(x)\neq0$ for all $x\in X$. Moreover, this can be achieved by a self-adjoint perturbation on $h^k$ by applying the adjoint perturbation on the other side of the diagonal. So from now on we assume, without loss of generality, that $h^k$ is such that the block $b'$ inside it, as defined in equation (\ref{bprime}), satisfies $b'(x)\neq0$ for all $x\in X$. This ensures that $b(x)\neq 0$ and that
$$c(x):=\frac{b(x)}{\Vert b(x)\Vert}+e_1\neq 0\;\;\;\text{for all}\;x\in X.$$
The latter equation defines a continuous map $c:X\rightarrow H-\{0\}$. Let $v_0(x)\in \unitaries(H)$ denote the orthogonal reflection around the line generated by $c(x)$. Notice that for all $x\in X$, $v_0(x)b(x)$ belongs to the span of $e_1$. Moreover, since $c$ is continuous, $v_0$ is norm-continuous. Now define another unitary field $v$ by letting, for all $x\in X$,
$$v(x)=\left(\begin{matrix} 1_{k-1} & 0 \\ 0 & v_0(x) \end{matrix}\right).$$
It follows that $u$ is still norm-continuous and that $vh^kv^*$ satisfies all the wanted properties for $h^{k+1}$. Let us define $u^{k+1}:=vu^k$, and $g^{k+1}=g^k$. That definition of $g^{k+1}$ may look surprising but the truth is that we allowed ourselves above to modify $h^k$ ``in place'', so of course it should be understood that $g^k$ was then modified accordingly to preserve the equality $h^k=u^kg^ku^{k*}$. At the next iteration, we have
$$h^{k+1}=u^{k+1}g^{k+1}u^{k+1*}=vu^kg^ku^{k*}v^*=vh^kv^*$$
and we checked above that $vh^kv^*$ has all the wanted properties for $h^{k+1}$.
\end{proof}

\begin{lemma}
\label{strongcont}
Let $X$ be a topological space. Let $(f^k)_{k\in \posint}$ be a sequence of strongly continuous maps from $X$ to $B(H)$. Suppose that there exists $M\in\continuous(X,\R_{>0})$ such that
$$\Vert f^k(x)\Vert \leq M(x)\;\;\;\text{for all $x\in X$ and all $k\in\posint$}.$$
Suppose that for all $k\in \posint$ there exists $l_0\in \posint$ such that for all $l\geq l_0$ the first $k$ columns of $f^l$ and of $f^{l_0}$ agree, in other words:
$$f^l_{ij}(x) = f^{l_0}_{ij}(x)\;\;\;\text{for all}\;x\in X, 1\leq j\leq k, 1\leq i.$$
It follows that for all $x\in X$, the sequence $(f^k(x))_{k\in \posint}$ converges strongly, and that, letting $f(x)$ denote its strong limit, the map $f$ thus defined on $X$ is strongly continuous.
\end{lemma}
\begin{proof}
For $k\in \posint$, let $p_k:H\rightarrow H$ denote the orthogonal projection onto the linear span of $e_1,\ldots,e_k$. Let $x\in X$. Let us check that $(f^k(x))_{k\in \posint}$ converges strongly. Let $\xi\in H$ and $\delta>0$. There exists $k\in \posint$ such that, letting $\xi^\prime=p_k(\xi)$, we have $\Vert \xi^\prime-\xi\Vert<\delta/M(x)$. For all $l\geq k$, we have
$f^l(x)\xi^\prime=f^k(x)\xi^\prime$,
whence $\Vert f^l(x)\xi - f^k(x)\xi^\prime \Vert < \delta$
showing that the sequence $(f^l(x)\xi)$ is Cauchy. This shows that $(f^l(x))_{l\in \posint}$ converges strongly. The strong limit $f(x)$ still satisfies $\Vert f(x)\Vert\leq M(x)$.\\

Let us now prove that $f$ is strongly continuous. Let $x_0\in X$. Let $\xi\in H$. Let $\delta>0$. There exists $k\in \posint$ such that, letting $\xi^\prime=p_k(\xi)$, we have
$$\Vert \xi^\prime-\xi\Vert<\frac{\delta}{8M(x_0)}.$$
By construction, for all $l\geq k$, we have $f^{l}(x)\xi^\prime=f^{k}(x)\xi^\prime$ for all $x\in X$. It follows that $f(x)\xi^\prime=f^{k}(x)\xi^\prime$.
Since $f^k$ is strongly continuous, this shows that the map $\alpha:x\mapsto f(x)\xi^\prime$ is continuous on $X$. Let $U$ be the set of all $x\in X$ such that $\Vert \alpha(x)-\alpha(x_0)\Vert <\delta/2$ and $M(x)<2M(x_0)$. Thus $U$ is an open neighborhood of $x_0$ of $X$. For all $x\in U$, we have
$$\Vert f(x)\xi - f(x_0)\xi \vert
\leq
 \Vert f(x)\xi - f(x)\xi^\prime \Vert
+\Vert \alpha(x) - \alpha(x_0) \Vert
+\Vert f(x_0)\xi^\prime - f(x_0)\xi \Vert
.$$
But for all $x\in U$, $$\Vert f(x)\xi - f(x)\xi^\prime \Vert \leq M(x)\Vert \xi^\prime-\xi\Vert \leq \delta/4.$$
It follows that for all $x\in U$,
$$\Vert f(x)\xi - f(x_0)\xi \Vert \leq \delta/4 + \delta/2 + \delta/4 \leq \delta.$$
This proves that $f$ is strongly continuous.
\end{proof}

Let us now finish the proof of our main theorem.

\begin{proof}[Proof of Theorem \ref{hessenbergoperators}]
Apply Lemma \ref{finitehessenberg} to obtain sequences $(g^k)_{k\in \posint}$ and $(u^k)_{k\in \posint}$ and let $h^k=u^kg^ku^{k*}$. By Lemma \ref{strongcont}, for all $x\in X$, the sequences $(u^{k*}(x))$ and $(h^k(x))$ converge strongly, and, letting $v(x)$ and $h(x)$ denote their respective strong limits, the maps $v,g,h$ thus defined on $X$ are strongly continuous (we will prove below that $h=v^*gv$).\\

Notice that for all $x\in X$, since $v(x)$ is an isometry because it is a strong limit of unitaries.\\

Since $\Vert g^{k+1}(x)-g^k(x)\Vert <\epsilon(x)/2^k$ and $g^{k+1}-g^k$ is self-adjoint, compact and norm-continuous for all $k$ and $g^0=f$, it follows that $(g^k-f)_{k\in \posint}$ converges locally uniformly to $g-f$ and that $g-f$ is self-adjoint, compact, norm-continuous, and $\Vert g(x)-f(x)\Vert < \epsilon(x)$ for all $x\in X$. This shows in particular that $g$ is strongly continuous.\\

Let us show that $h=v^*gv$. It is enough to show that $(u^kg^ku^{k*})(x)$ converges weakly to $(v^*gv)(x)$ for all $x\in X$. Let $\xi,\eta\in H$. We have
$$\scalprod{(u^kg^ku^{k*})(x)\xi,\eta}=\scalprod{(g^ku^{k*})(x)\xi,(u^{k*})(x)\eta}.$$
As the sequences $(g^{k})$ and $(u^{k*})$ converge strongly to $g$ and $v$ respectively, and the strong topology makes multiplication jointly continuous on bounded sets, the sequence $(g^ku^{k*})$ also converges strongly to $gv$, and we obtain
$$\scalprod{(u^kg^ku^{k*})(x)\xi,\eta}\rightarrow\scalprod{(gv)(x)\xi,v(x)\eta}=\scalprod{(v^*gv)(x)\xi,\eta}.$$
This proves that $h = v^*gv$.
\end{proof}

\def\cprime{$'$}

\end{document}